\documentclass[11pt,reqno]{amsart}
\usepackage{amssymb,mathrsfs,graphicx,amsmath,mathtools}
\usepackage{ifthen}
\usepackage{hyperref}
\usepackage[margin=1.2in]{geometry}
\usepackage{caption}
\usepackage{sidecap}
\usepackage{rotating}
\usepackage{color}

\provideboolean{shownotes} 
\setboolean{shownotes}{true} 
\newcommand{\margnote}[1]{
\ifthenelse{\boolean{shownotes}}%
{\marginpar{\raggedright\tiny\texttt{#1}}}%
{}%
}
\newcommand{\hole}[1]{
\ifthenelse{\boolean{shownotes}}%
{\begin{center} \fbox{ \rule {.25cm}{0cm} \rule[-.1cm]{0cm}{.4cm}
\parbox{.85\textwidth}{\begin{center} \texttt{#1}\end{center}} \rule
{.25cm}{0cm}}\end{center}} {} }

\title[Global existence to self-consistent chemotaxis-fluid systems]{Global existence and decay rates \\
to self-consistent chemotaxis-fluid system}

\author[Carrillo]{Jose A. Carrillo}
\address[Jose A. Carrillo]{\newline Mathematical Institute, University of Oxford \newline
Andrew Wiles Building, Radcliffe Observatory Quarter, Woodstock Road, OX2 6GG, UK}
\email{carrillo@maths.ox.ac.uk}

\author[Peng]{Yingping Peng}
\address[Yingping Peng]{\newline School of Mathematics \newline
Southwest Jiaotong University, Chengdu 611756, China}
\email{yingpingpeng@swjtu.edu.cn}

\author[Xiang]{Zhaoyin Xiang}
\address[Zhaoyin Xiang]{\newline School of Mathematical Sciences\newline
University of Electronic Science and Technology of China, Chengdu 611731, China}
\email{zxiang@uestc.edu.cn}

\numberwithin{equation}{section}

\newtheorem{theorem}{Theorem}[section]
\newtheorem{lemma}{Lemma}[section]

\newtheorem{proposition}{Proposition}[section]
\newtheorem{remark}{Remark}[section]
\newtheorem{definition}{Definition}[section]

\newcommand{\R}{\mathbb R}

\newcommand{\om}{\omega}

\newcommand{\bq}{\begin{equation}}
\newcommand{\eq}{\end{equation}}

\newcommand{\lt}{\left}
\newcommand{\rt}{\right}

\newcommand{\pa}{\partial}
\newcommand{\na}{\nabla}
\newcommand{\cd}{\cdot}
\newcommand{\ep}{\varepsilon}
\newcommand{\Del}{\Delta}
\newcommand{\ga}{\gamma}
\newcommand{\f}{\frac}

\newcommand{\bfu}{\mathbf{u}}
\newcommand{\bfv}{\mathbf{v}}
\newcommand{\bfw}{\mathbf{w}}

\newcommand{\nn}{\notag}

\begin{document}
\allowdisplaybreaks

\date{\today}

\subjclass[]{}

\begin{abstract}
In this paper, we investigate a chemotaxis-fluid system  involving both the effect of potential force on cells and the effect of chemotactic force on fluid:
\begin{equation*}
\left\{
\begin{split}
\pa_t n  +  \bfu\cd\na n & =  \Del n  -  \na\cd\left(\chi(c)n\na c\right)  +  \na\cd(n\na\phi),   \\
\pa_t c  +  \bfu\cd\na c  &=  \Del c  -  nf(c),  \\
\pa_t\bfu + \kappa(\bfu\cd\na)\bfu  + \na P & =  \Del\bfu - n\na\phi  +  \chi(c)n\na c,  \\
\na\cd\bfu &= 0
\end{split}
\right. 
\end{equation*}
in $\R^d\times(0,T)\, (d=2,3)$. One of the novelties and difficulties here is that the coupling in this model is stronger and more nonlinear than the most-studied chemotaxis-fluid model.   We  will first establish several extensibility criteria of classical solutions, which ensure us to extend the local solutions to global ones in the three dimensional chemotaxis-Stokes case and in the two dimensional chemotaxis-Navier-Stokes version under suitable smallness assumption on $\|c_0\|_{L^{\infty}}$ with the help of a new entropy functional inequality. Some further decay estimates are also obtained under some suitable growth restriction on the potential $\nabla \phi$ at infinity.  As a byproduct of the entropy functional inequality, we also establish the global-in-time existence of weak solutions to the three dimensional chemotaxis-Navier-Stokes system.  To the best of our knowledge, this  seems to be the first work addressing the global well-posedness and decay property of solutions to the Cauchy problem of self-consistent chemotaxis-fluid system. 

\end{abstract}

\maketitle \centerline{\date}

\section{Introduction}

\subsection{Background \& literature review}

In nature, most living cells or organisms (e.g. \textit{Dictyostelium}, \textit{Bacillus subtilis}, \textit{Escherichia coli}, etc.) are endowed with the ability to sense certain stimulating chemical signals (e.g. nutrients) in the environment (mostly the viscous fluid) and adapt their movements accordingly. A significant number of experiments and analytical studies have revealed noticeable facts that organisms and the complex surrounding fluid may substantially affect the motion of them each other through e.g. chemotactic forces and gravitational forces by organisms, buoyant forces by fluid, etc. For example, to describe this kind of cell-fluid interactions mathematically, Tuval et al. \cite{TCD2005PNAS} conducted a highly influential experiment in a water drop suspended with swimming \textit{Bacillus subtilis}, and proposed the following chemotaxis-fluid model
\begin{equation}\label{EQ-Tuval}
\left\{
\begin{split}
\pa_t n  +  \bfu\cd\na n  & =  \lambda\Del n  -  \na\cd\big(\chi(c)n\na c\big),   \\
\pa_t c  +  \bfu\cd\na c  & =  \nu\Del c  -  nf(c),  \\
\pa_t\bfu + \kappa(\bfu\cd\na)\bfu  + \na P  & = \mu \Del\bfu - n\na\phi,  \\
\na\cd\bfu &= 0, \qquad \qquad x\in\Omega,\,\,\, t>0.
\end{split}
\right.
\end{equation}
Here the time evolution of cell density $n=n(x, t)$ is described by \eqref{EQ-Tuval}$_1$, in which the chemotaxis-induced aggregation with chemotactic sensitivity $\chi$, diffusion with strength $\lambda$ caused by random Brownian motion, and transportation by ambient fluid $\bfu=\bfu(x, t)$ subjected to incompressible (Navier-)Stokes equation with pressure $P=P(x,t)$ are considered.  The signal with concentration $c=c(x, t)$ also diffuses with strength $\nu$, is transported by ambient fluid and is consumed by cells with consumption rate $f$. The external force $-n\na\phi$ exerted on the fluid by cells can be produced by various physical mechanisms, e.g. gravity, centrifugal, electric forces, magnetic forces, etc. We refer to \cite{Lorz2010M3AS,TCD2005PNAS} for more details concerning its physical background and \cite{Bellomo2016M3} for a different derivation (under the framework of the kinetic theory) of chemotaxis-fluid models.

The fundamental mathematical challenges arising in the analysis of \eqref{EQ-Tuval} appear to consist in two parts, i.e. one is the nonlinear cross-diffusion term $-\na\cd\big(\chi(c)n\na c\big)$ which may enforce blow-up solutions (see \cite{Bellomo2015M3} and references therein), and the other one arises from the well-known incompleteness theory of existence and regularities with large initial data for the Navier-Stokes system. Since the seminal analytical work \cite{Lorz2010M3AS} proved the local existence of weak solutions, there have been plenty of scholars devoting themselves to the well-posedness and long-time behavior of solutions to \eqref{EQ-Tuval}. For instance,  Duan et al. \cite{Duan2010CPDE}  proved the global-in-time existence and explicit decay rates of classical solutions for the chemotaxis-Navier-Stokes system near constant state $(n_{\infty}, 0, \mathbf{0})$ in $\R^3$, and also established global-in-time existence of weak solutions for the chemotaxis-Stokes system in $\R^2$ provided that the external forcing $\phi$ is weak or the signal concentration $c$ is small, i.e.
\begin{align}\label{weak-phi}
\phi\geq0,\,\, \na\phi\in L^{\infty}(\R^2),\,\,
\sup_{x\in\R^2}\left(\om(x)|\na\phi(x)|+\om^2(x)|\na^2\phi(x)|\rt)\,\,\textrm{and}\,\,\|c_0\|_{L^4(\R^2)}\,\,  \textrm{are small},
\end{align}
or
\begin{equation}\label{weak-c}
\phi\geq0,\, \na\phi\in L^{\infty}(\R^2), \,
\sup_{x\in\R^2}\left(\om(x)|\na\phi(x)|+\om^2(x)|\na^2\phi(x)|\rt)<\infty,\, \,\|c_0\|_{L^{\infty}(\R^2)}\,\,\textrm{is small},
\end{equation}
where $\om(x)=(1+|x|)(1+\ln(1+|x|))$. In \cite{Liu2011AIHA}, Liu-Lorz got rid of the decay of the potential $\phi$ at infinity and the smallness of $c_0$ in \eqref{weak-phi} and \eqref{weak-c} by making technical assumptions on $\chi$ and $f$:
\begin{align}\label{chi-f-1}
\na\phi\in L^{\infty}(\R^2),\quad  \chi, \, \chi', \, f, \, f'\geq0,\quad  \f{d^2}{dc^2}\lt(\f{f(c)}{\chi(c)}\rt)<0,\quad  \f{\chi' f + \chi f'}{\chi}>0,
\end{align}
and obtained the global existence of weak solutions even for the full chemotaxis-Navier-Stokes system. Then Chae et al. \cite{Chae13DCDS} presented some blow-up criteria for the local solutions to the chemotaxis-Navier-Stokes system in $\R^d \,(d=2,3)$, and proved the global existence of classical solutions in $\R^2$ on quite different assumption from \eqref{chi-f-1} that for some constant $\mu$
\begin{align}\label{chi-f-2}
\phi, \, \chi, \, \chi', \, f, \, f'\geq0,\quad  \sup_{c}|\chi(c)-\mu f(c)|<\varepsilon\,\,\,\,  \textrm{for a sufficiently small} \,\,\varepsilon>0.
\end{align}
In particular, if the last condition on $\chi$ and $f$ in \eqref{chi-f-2} is replaced by $\chi(c)-\mu f(c)=0$, then the global existence of weak solutions was also showed in $\R^3$ in \cite{Chae13DCDS}. In \cite{Chae14CPDE}, they further established the global existence of classical solutions and explicit temporal decay rates under the smallness assumptions on initial data. For more rigorous analytical results for the Cauchy problems of the coupled chemotaxis-fluid system \eqref{EQ-Tuval}, we refer to \cite{Duan2014IMRN, Li2016JDE, Zhang2014SIAM, Duan2017JDE} and references therein. Additionally, as far as the physical domain is concerned, various types of initial-boundary value problems (in bounded domains or unbounded domains with finite depth) for system \eqref{EQ-Tuval} are extensively studied (see e.g. \cite{Winkler2012CPDE, Winkler2014ARMA, Winkler2021CPDE, Peng2018M3, Peng2019JDE} and references therein). We also remark that apart from the model \eqref{EQ-Tuval} itself, a large amount of significant variants acting as an additional regularizing mechanism for preventing possible blow-up solutions have been widely studied as well. For instance, taking the nonlinearly enhanced cell diffusion at large densities into consideration by replacing $\Del n$ in \eqref{EQ-Tuval}$_1$ with porous medium-type $\Del n^m$ for $m>1$, we refer to \cite{Francesco2010DCDS,Duan2014IMRN,Tao2012DCDS,Liu2011AIHA} for related Cauchy problems and \cite{Ishida2015DCDS,Tao2012DCDS,Zhang2015JDE} for initial-boundary value problems.

While system \eqref{EQ-Tuval} and its relative variants presented above have been well studied and understood numerically (\cite{TCD2005PNAS,Chertock2012JFM,Lee2015EJM}) and analytically, Lorz \cite{Lorz2010M3AS} and Di Francesco et al. \cite{Francesco2010DCDS} pointed out that it could be more realistic to include both the effect of gravity (potential force) on cells and the effect of the chemotactic force on fluid, i.e. they extended the system \eqref{EQ-Tuval} to the following self-consistent version
\begin{equation}\label{EQ-m}
\left\{
\begin{split}
\pa_t n  +  \bfu\cd\na n  & =  \Del n^m  -  \na\cd\big(\chi(c)n\na c\big)  +  \na\cd(n\na\phi),   \\
\pa_t c  +  \bfu\cd\na c  & =  \Del c  -  nf(c),  \\
\pa_t\bfu + \kappa(\bfu\cd\na)\bfu  + \na P  & =  \Del\bfu - n\na\phi  +  \chi(c)n\na c,  \\
\na\cd\bfu & = 0,  \qquad \qquad x\in\Omega,\,\,\, t>0.
\end{split}
\right.
\end{equation}
The reasoning behind the coupling $\chi(c)n\na c$ in equation \eqref{EQ-m}$_3$ is that the fluid will exert frictional force on the moving cells to make the cells move without acceleration and thus that the reaction forces act on the fluid, which also matches the nonlinear cross-diffusion term in the cell density in equation \eqref{EQ-m}$_1$.  We would like to remark that the similar forcing $n\na c$ was also appeared in the coupling Nernst-Planck-Navier-Stokes system (see \cite{Constantin2019ARMA} and references therein).  

In contrast to the large amount of existing works on system \eqref{EQ-Tuval} and its variants, the researches on the  well-posedness of system \eqref{EQ-m} are far fewer. When $\Omega\subset\R^2$ is a bounded domain, Di Francesco et al. \cite{Francesco2010DCDS} proved the existence of global weak solution to the no-flux/no-flux/no-slip boundary value problem of system \eqref{EQ-m} with $\kappa=0$ and $\f32<m\leq2$, which was extended to $m>1$ by  Yu \cite{Yu2020M2} for $\kappa=0$ again and by Wang  \cite{Wang2020DCDS} for general $\kappa\in\R$, respectively.  In the three-dimensional setting,  the similar  global existence  was also obtained in case $\kappa=0$ and $m>\f43$ by \cite{Wang2020JDE}.  To the best of our knowledge, the only available result for the linear diffusion case $m=1$ is due to Lorz \cite{Lorz2010M3AS}, where the local existence of weak solutions to system \eqref{EQ-m} with $\kappa=0$ in a planar bounded domain was established.

\subsection{Main results} 

Motivated by above works, we concern in this paper with the Cauchy problem for the self-consistent chemotaxis-(Navier-)Stokes model with linear cell diffusion
\begin{equation}\label{EQ}
\left\{
\begin{split}
\pa_t n  +  \bfu\cd\na n  & =  \Del n  -  \na\cd\left(\chi(c)n\na c\right)  +  \na\cd(n\na\phi),   \\
\pa_t c  +  \bfu\cd\na c  & =  \Del c  -  nf(c),  \\
\pa_t\bfu + \kappa(\bfu\cd\na)\bfu  + \na P & =  \Del\bfu - n\na\phi  +  \chi(c)n\na c,  \\
\na\cd\bfu & = 0
\end{split}
\right.
\end{equation}
in the space-time region $\R^d\times (0,T)$ with $d=2, 3$, which will be supplemented with the initial conditions
\begin{align}\label{initial}
(n, c, \bfu)\big|_{t=0} = (n_0, c_0, \bfu_0)  \qquad \text{in}\quad \R^d. 
\end{align}
To state our results precisely,  we assume basically that
\begin{enumerate}
  \item[(\textbf{A}):] $\chi, \, f\in C^1\big([0, +\infty)\big)$ with $f(0)=0$ and $f(s)\geq 0$ for all $s\ge 0$;
  \item[(\textbf{B}):] $\na\phi\in L^{\infty}(\R^d)$;
  \item[(\textbf{C}):] The initial data $(n_0, c_0, \bfu_0)$ satisfy that
  \begin{align*}
  n_0\geq0, \qquad c_0\geq0, \qquad \na\cd\bfu_0=0  \qquad \text{in}\quad \R^d
  \end{align*}
  and that
  \[
  n_0(1 + |x| + |\ln n_0|)\in L^1(\R^d),  \quad
  c_0 \in L^1(\R^d)\cap L^{\infty}(\R^d)\cap H^1(\R^d), \quad \bfu_0\in H^1(\R^d;\R^d).
  \]
\end{enumerate}

Our first main result is concerning the Serrin-type extensibility criteria  in three dimensional case.

\begin{theorem}[Extensibility criteria  in $\R^3$]\label{blowup1}
Let the initial data $(n_0,c_0,\bfu_0)\in H^{m-1}(\R^3)\times H^m(\R^3)\times H^m(\R^3:\R^3)$ with  $m \ge 3$ satisfy the assumptions $(\bf{A})$, $(\bf{B})$ and $(\bf{C})$  with $d=3$. Suppose that $\chi, \,  f \in C^m\big([0,+\infty)\big)$ and $\phi\in W^{m,\infty}(\R^3)$. If the maximal existence time $T^*$ of system \eqref{EQ}-\eqref{initial} is finite, then in the case  $\kappa\neq0$,  it holds that 
\begin{equation}\label{blow-1}
\int_0^{T^*}\|\na c\|^{s_1}_{L^{r_1}(\R^3)}   +   \|\bfu\|^{s_2}_{L^{r_2}(\R^3)} dt =  +\infty
\end{equation}
for any $(r_i, s_i)$ satisfying $\f{3}{r_i} + \f{2}{s_i}\le 1$ and $3<r_i \le +\infty \, (i=1, 2)$, while in the case $\kappa=0$, it holds that 
\begin{equation}\label{blow-1-1}
\int_0^{T^*}\|\na c\|^{s_1}_{L^{r_1}(\R^3)}    dt =  +\infty
\end{equation}
for any $(r_1, s_1)$ satisfying $\f{3}{r_1}+\f{2}{s_1}\leq1$ and $3<r_1\leq+\infty$.
\end{theorem}

The similar results also hold in the two dimensional setting.

\begin{theorem}[Extensibility criteria in $\R^2$]\label{blowup2}
Let the initial data $(n_0, c_0, \bfu_0)\in H^{m-1}(\R^2)\times H^m(\R^2)\times H^m(\R^2: \R^2)$ for $m\geq3$ satisfy the assumptions $(\bf{A})$, $(\bf{B})$ and $(\bf{C})$  with $d=2$. Suppose that $\kappa\in\R$, $\chi, \, f \in C^m\big([0,+\infty)\big)$ and $\phi\in W^{m, \infty}(\R^2)$. If the maximal   existence time $T^*$ of system \eqref{EQ}-\eqref{initial} is finite, then
\begin{equation}\label{blow-2}
\int_0^{T^*}\|\na c\|^{s_3}_{L^{r_3}(\R^2)}  dt  =  +\infty
\end{equation}
for any $(r_3,s_3)$ satisfying $\f{2}{r_3}+\f{2}{s_3}\leq1$ and $2<r_3\le +\infty$.
\end{theorem}

With the help of these extensibility criteria,  we can extend the local solutions to global ones for suitably small  $c_0$.

\begin{theorem}[Global existence of classical solutions in $\R^3$]\label{Stokes}
 Suppose that all assumptions in Theorem \ref{blowup1} hold. If additionally $\|c_0\|_{L^{\infty}(\R^3)}$ is suitably small, then the unique classical solution $(n, c, \bfu)$ of system \eqref{EQ}-\eqref{initial}  in $\R^3$ with $\kappa=0$ exists globally in time and satisfies that for any $T<\infty$,  
\[
(n,c, \bfu)\in L^{\infty}\big(0,T; H^{m-1}(\R^3)\times H^m(\R^3)\times H^m(\R^3;\R^3)\big)
\]
and 
\[
(\na n,\na c, \na\bfu)\in L^2\big(0,T; H^{m-1}(\R^3)\times H^m(\R^3)\times H^m(\R^3;\R^3)\big).
\]
\end{theorem}

\begin{remark}\label{3D-NS}
Theorem \ref{Stokes} shows the global existence of classical solution to system \eqref{EQ}-\eqref{initial} with $\kappa=0$. As for the case $\kappa\neq0$, we can only establish the global existence of weak solutions due to the well-known challenge in the Navier-Stokes equations. We will  postpone the proof of such global weak solutions to the appendix section.
\end{remark}

In the two dimensional case, similar global well-posedness can be established even for  the chemotaxis system coupled by the Navier-Stokes equations.

\begin{theorem}[Global existence of classical solution in $\R^2$]\label{2D-NS}
 Suppose that all assumptions in Theorem \ref{blowup2} hold. If additionally $\|c_0\|_{L^{\infty}(\R^2)}$ is suitably small, then the unique classical solution $(n, c, \bfu)$ of  system \eqref{EQ}-\eqref{initial}  in $\R^2$ with $\kappa\in\R$ exists globally in time and satisfies that  for any $T<\infty$, 
\[
(n,c, \bfu)\in L^{\infty}\big(0,T; H^{m-1}(\R^2)\times H^m(\R^2)\times H^m(\R^2;\R^2)\big) 
\]
and 
\[
(\na n,\na c, \na\bfu)\in L^2\big(0,T; H^{m-1}(\R^2)\times H^m(\R^2)\times H^m(\R^2;\R^2)\big).
\]
\end{theorem}

Finally, for the classical solution $(n,c,\bfu)$ obtained in Theorem \ref{Stokes} and Theorem \ref{2D-NS}, we can establish some temporal decay estimates. For notational simplicity, we denote
\[
\om(x):=\left\{
\begin{split}
&|x|,  \qquad x\in\R^3, \\
&(1+|x|)(1+\ln(1+|x|)), \,\,\, x\in\R^2
\end{split}
\right.
\]
and
\begin{equation}\label{omega}
\quad \mathcal{M}_{\om\phi}:=\sup_{x\in\R^d}\big(|\om(x)|\,|\na\phi(x)|\big)^2.
\end{equation}

\begin{theorem}[Decay estimates]\label{decay}
 Suppose that all assumptions in  Theorem \ref{Stokes} and Theorem \ref{2D-NS} hold.  If additionally  $\mathcal{M}_{\om\phi}$ defined by \eqref{omega}  is suitably small,  then the solution $(n, c, \bfu)$ of system \eqref{EQ}-\eqref{initial} enjoys the following temporal decay: for any $1\leq p<\infty$, it holds that
\begin{align}\label{decay:n}
\|n(t)\|_{L^p(\R^d)}\leq C(\|n_0\|_{L^1(\R^d)\cap L^p(\R^d)},\|c_0\|_{L^{\infty}(\R^d)})\,(1+t)^{-\f{d}{2}(1-\f{1}{p})},
\end{align}
and
\begin{align}\label{decay:c}
\|c(t)\|_{L^p(\R^d)}\leq C(\|c_0\|_{L^1(\R^d)\cap L^p(\R^d)})\,(1+t)^{-\f{d}{2}(1-\f{1}{p})};
\end{align}
furthermore, for the signal concentration, it also holds that
\begin{align}\label{decay:nc}
\|c(t)\|_{L^{\infty}(\R^d)}  \leq C(\|c_0\|_{L^2(\R^d)\cap L^{\infty}(\R^d)})(1+t)^{-\f{d}{4}}.
\end{align}
\end{theorem}

\subsection{Main ideas and structure of the paper} 

We will first establish the extensibility criteria (Theorem \ref{blowup1} and Theorem \ref{blowup2}) in Section \ref{sec:blowup}. Here we would like to remark  that if we specially choose $s_1=2$, $r_1=\infty$ and $s_2=q$, $r_2=p$ with $\f{3}{p}+\f{2}{q}=1$, $3<p\leq\infty$ in Theorem \ref{blowup1}, and choose $s_3=2$, $r_3=\infty$ in Theorem \ref{blowup2},  we can cover the corresponding extensibility criteria obtained in \cite[Theorem 1.2]{Chae13DCDS} for system \eqref{EQ-Tuval}. In comparison, our extensibility  criteria \eqref{blow-1}-\eqref{blow-2} in Theorem \ref{blowup1} and Theorem \ref{blowup2} are relatively easier to be verified. Indeed, the bounds   
\[
\int_0^{T^*}\|\na c\|_{L^5(\R^3)}^5 + \|\bfu\|_{L^5(\R^3)}^5 dt<\infty
\]
in three dimensional case and 
\[
\int_0^{T^*}\|\na c\|_{L^4(\R^2)}^4dt<\infty
\]
in two dimensional setting are enough to extend the local solutions to global ones and to establish the global existence of classical solutions (Theorem \ref{Stokes} and Theorem \ref{2D-NS}) with the help of  the entropy functional inequality 
\begin{align*}
& \int_{\R^d} n|\ln n|   +  \|\na c\|_{L^2(\R^d)}^2  +  \|\bfu\|_{L^2(\R^d)}^2   \nn \\
&\quad +  \f12\int_0^t\Big(\int_{\R^d}\f{|\na n|^2}{n}  + \|\na^2 c\|_{L^2(\R^d)}^2   +  \|\na\bfu\|_{L^2(\R^d)}^2 \Big) ds
\leq C, 
\end{align*} 
which will be exhibited in Section \ref{globalexist}.  The key to achieve this relies on deriving an entropy functional inequality. As is showing in the equation \eqref{EQ}$_3$, the appearance of the coupling term $\chi(c)n\na c$ leads to the more stronger nonlinearity, making it difficult to close the entropy estimate. To overcome this difficulty, we will  increase the integrability of $n$ from $L^1(\R^d)$ to $L^p(\R^d)$ for $p\in[1,\infty)$ under suitable smallness assumption on $\|c_0\|_{L^{\infty}(\R^d)}$.  We will next give the explicit decay rates for the regular solutions (Theorem \ref{decay}) in Section \ref{decayestimate} by suitably restricting the growth of the potential $\na\phi$ at infinity. This extra assumption will remove the obstacle arising from  the lack of inform-in-time estimate of $\|\na c\|_{H^1(\R^d)}$  by introducing a weighted function $g=e^{(\beta c)^2}$ with $\beta>0$. Finally, we will give a sketch for the proof of global existence of weak solutions to the three dimensional chemotaxis-Navier-Stokes system in Section \ref{appendix}.

\subsection{Notation} We will set  $\pa_t=\f{\pa}{\pa t}$ and  $\pa_i=\f{\pa}{\pa x_i}$ for $i=1,2, \cdots, d$,  and denote all the partial derivatives $\pa^{\alpha}$ with multi-index $\alpha$ satisfying $|\alpha|=k$ by  $\na^k$ ($k\geq0$). Let  $C=C(\alpha, \beta, \cdots)$ be a generic positive constant depending only on $\alpha,\beta,\cdots$ but not on  $\kappa$.

\vspace{2mm}

\section{Preliminaries}\label{sec:rel}

In this section, we would like to present some preliminaries. We begin with recalling the well-known tame estimate for the product of two functions and  Moser estimate.

\begin{lemma}[Corollary 2.54 in \cite{Bahouri}]\label{leibniz}
  Let $k\in\mathbb{N}$. Then for any functions $f$, $g\in(H^k\cap L^{\infty})(\R^d)$,  there exists a positive constant $C=C(k,d)$ such that
  \[
  \big\|\nabla^k(fg) \big\|_{L^2(\R^d)}
  \le  C\Big(\|f\|_{L^{\infty}(\R^d)} \big\|\nabla^k g \big\|_{L^2(\R^d)}   + \big\|\nabla^k f \big\|_{L^2(\R^d)} \|g\|_{L^{\infty}(\R^d)}\Big).
 \]
  Furthermore, if $\na f\in L^{\infty}(\R^d)$, it holds that
  \[
  \big\|\na^k(fg) - f\na^k g \big\|_{L^2(\R^d)}
  \le  C\Big(\|\na f\|_{L^{\infty}(\R^d)}  \big\|\na^{k-1}g  \big\|_{L^2(\R^d)}
     + \big\|\na^k f \big\|_{L^2(\R^d)}\|g\|_{L^{\infty}(\R^d)}\Big).
  \]
  \end{lemma}

 \begin{lemma}[Theorem 2.61 in \cite{Bahouri}]\label{moser}
 Let $k\in\mathbb{N}$, $p\in[1,\infty]$ and $f\in C^k(\R^d)$. Then there exists a positive constant $C=C(k, p,f)$ such that
 \[
  \big\|\na^k f(\omega) \big\|_{L^p(\R^d)}  \le  C\|\omega\|_{L^{\infty}(\R^d)}^{k-1} \big\|\na^k\omega \big\|_{L^p(\R^d)}
 \]
  for all $\omega\in(W^{k,p}\cap L^{\infty})(\R^d)$.
\end{lemma}

We now state the local-in-time existence of classical solutions to the Cauchy problem \eqref{EQ}-\eqref{initial} and give a sketch for its proof.

\begin{lemma}[Local well-posedness]\label{local}

Let the assumptions $(\bf{A})$, $(\bf{B})$ and $(\bf{C})$ hold. Suppose that  $\kappa\in\R$, $\chi$, $f\in C^m\big([0, +\infty)\big)$ and  $\phi\in W^{m,\infty}(\R^d)$ with $d=2,3$ and $m\geq3$. If the initial data $(n_0, c_0, \bfu_0)\in H^{m-1}(\R^d)\times H^m(\R^d)\times H^m(\R^d;\R^d)$, then there exist $T^* \in(0,+\infty]$ and a unique triple $(n, c, \bfu)$ fulfilling that for any $t<T^*$,
\[
(n,c, \bfu)\in L^{\infty}\big(0, t; H^{m-1}(\R^d)\times H^m(\R^d)\times H^m(\R^d;\R^d)\big)
\]
and
\[
(\na n,\na c, \na\bfu)\in L^2\big(0,t; H^{m-1}(\R^d)\times H^m(\R^d)\times H^m(\R^d;\R^d)\big),
\]
and solving system \eqref{EQ}-\eqref{initial}.
\end{lemma}

\begin{proof}
This lemma can be shown by following the proof of Lemma 2.2 in \cite{Duan2010CPDE}.  Here we just mention that we can construct the approximate  solution sequence $(n^j, c^j, \bfu^j)_{j\geq0}$ by iteratively solving the  linear  Cauchy problems
\begin{equation*}
\left\{
\begin{split}
\pa_t n^{j+1}  +  \bfu^{j}\cd\na n^{j+1} & =  \Del n^{j+1}  -  \na\cd\left(\chi(c^{j})n^{j+1}\na c^{j}\right)  +  \na\cd(n^{j}\na\phi),   \\
\pa_t c^{j+1}  +  \bfu^{j}\cd\na c^{j+1}  &=  \Del c^{j+1}  -  n^{j}f(c^{j}),  \\
\pa_t\bfu^{j+1} + \kappa(\bfu^{j}\cd\na)\bfu^{j+1}  + \na P^{j+1} & =  \Del\bfu^{j+1} - n^{j}\na\phi  +  \chi(c^{j})n^{j}\na c^{j},  \\
\na\cd\bfu^{j+1} &= 0  \\
\big(n^{j+1}(x,0),c^{j+1}(x,0),\bfu^{j+1}(x,0)\big)&=\big(n_0(x),c_0(x),\bfu_0(x)\big)
\end{split}
\right.
\end{equation*}  
with the first iterative step $\big(n^0(x,t),c^0(x,t),\bfu^0(x,t)\big)=\big(n_0(x),c_0(x),\bfu_0(x)\big)$.
\end{proof}

Then the following lemma shows several basic properties of solutions to system \eqref{EQ}-\eqref{initial}, e.g., the conservation of mass and the maximum principle.
\begin{lemma}\label{mass}
Suppose that the assumptions in Lemma \ref{local} hold. Then the solution $(n, c, \bfu)$ of system \eqref{EQ}-\eqref{initial} satisfies
\[
n(x,t)\geq0, \quad c(x,t)\geq0 \qquad \textrm{a.e.} \,\,\, \mathrm{in} \quad  \R^d\times[0, +\infty),
\]
\[
\|n(t)\|_{L^1(\R^d)}=\|n_0\|_{L^1(\R^d)} \qquad \mathrm{for\,\,\, any }\quad  t\geq0
\]
and
\[
\sup_{t\geq0}\|c(t)\|_{L^{p}(\R^d)} \leq \|c_0\|_{L^{p}(\R^d)} \qquad \mathrm{for\,\,\, any }\quad t\geq0 \,\,\,\,\mathrm{and} \,\,\,\, 1\le p \le \infty.
\]
\end{lemma}
\begin{proof}
The proofs are pretty standard and we may refer to \cite[Lemma 2.1]{Duan2017JDE} for details.
\end{proof}

Based on the assumption (\textbf{A}) and the boundedness of $c$ presented in Lemma \ref{mass}, we introduce the following notations for simplicity
\[
c_{\infty}:=\|c_0\|_{L^{\infty}(\R^d)},
\quad \mathcal{C}_{f}:=\sup_{0\leq c\leq c_{\infty}}\big(|f(c)|+|f'(c)\big),
\quad \mathcal{C}_{\chi}:=\sup_{0\le c\leq c_{\infty}}\big(|\chi(c)|+|\chi'(c)|\big).
\]

\vspace{2mm}

\section{Extensibility Criterions}\label{sec:blowup}

In this section, we devote ourselves to establish some blow-up criteria, which will be significant tools for the proof of global existence.

As is known that the local existence result is a natural blow-up criteria. That is to say, if the maximal time $T^*$ of existence obtained in Lemma \ref{local} is finite, then
\begin{align}\label{blowup-local}
&\sup_{0\leq t\leq T^*}\Big(\|n(t)\|_{H^{m-1}(\R^d)}^2  + \|c(t)\|_{H^{m}(\R^d)}^2 + \|\bfu(t)\|_{H^{m}(\R^d)}^2  \Big)  \nn \\
&\qquad +\int_0^{T^*} \Big(\|\na n(t)\|_{H^{m-1}(\R^d)}^2  + \|\na c(t)\|_{H^{m}(\R^d)}^2 + \|\na\bfu(t)\|_{H^{m}(\R^d)}^2 \Big)dt = \infty
\end{align}
for $m\geq3$ and $d=2, 3$.

\subsection{Extensibility criterion in $\R^3$}

With the blow-up criteria \eqref{blowup-local} at hand, we are able to accomplish the proof of Theorem \ref{blowup1} in the following.

\begin{proof}[Proof of Theorem \ref{blowup1}]
We will prove this Theorem by contradictory arguments. Suppose that the assertions in Theorem \ref{blowup1} are not true, i.e.
\begin{equation}\label{blow:no-1}
(\kappa\neq0)\qquad\int_0^{T^*}\|\na c\|^{s_1}_{L^{r_1}(\R^3)}   +   \|\bfu\|_{L^{r_2}(\R^3)}^{s_2} dt <  +\infty
\end{equation}
for some $(r_i,s_i)$ satisfying $\f{3}{r_i}+\f{2}{s_i}\leq1$ and $3<r_i\leq+\infty$ ($i=1,2$), and
\begin{equation}\label{blow:no-2}
(\kappa=0)\qquad\int_0^{T^*}\|\na c\|^{s_1}_{L^{r_1}(\R^3)}  dt <  +\infty
\end{equation}
for some $(r_1,s_1)$ satisfying $\f{3}{r_1}+\f{2}{s_1}\leq1$ and $3<r_1\leq+\infty$. We now show that 
 the assumption $T^*<+\infty$  will lead to a contradiction to the local well-posedness result in Lemma \ref{local}.
 
\vspace{2mm}
\textbf{Step 1}: Estimate of $(n, c, \bfu)$ in $L^{\infty}_tL_x^2\times L^{\infty}_tH_x^1\times L^{\infty}_tH_x^1$.

Firstly, multiplying \eqref{EQ}$_1$ by $n$, integrating by parts, and using H\"{o}lder's inequality and Young's inequality, one has
\begin{align}\label{blow:nL2}
\f12\f{d}{dt}\|n\|_{L^2(\R^3)}^2  +  \|\na n\|_{L^2(\R^3)}^2  
& =  \int_{\R^3}\chi(c)n\na c\cd\na n   -  \int_{\R^3}n\na \phi\cd\na n  \nn \\
& \le  \f18\|\na n\|_{L^2(\R^3)}^2  +  4\mathcal{C}_{\chi}^2\|n\na c\|_{L^2(\R^3)}^2  +  4\|n\na\phi\|_{L^2(\R^3)}^2.
\end{align}
It follows from H\"{o}lder's inequality again and the Gagliardo-Nirenberg inequality that
\begin{align}\label{blow:nc}
\|n\na c\|_{L^2(\R^3)}^2
\le \|n\|_{L^{\f{2r_1}{r_1-2}}(\R^3)}^2  \|\na c\|_{L^{r_1}(\R^3)}^2
\le C\|\na n\|_{L^2(\R^3)}^{\f{6}{r_1}}  \|n\|_{L^2(\R^3)}^{\f{2(r_1-3)}{r_1}}  \|\na c\|_{L^{r_1}(\R^3)}^2.
\end{align}
Substituting \eqref{blow:nc} into \eqref{blow:nL2} and using Young's inequality, we have
\begin{align*}
&\f12\f{d}{dt}\|n\|_{L^2(\R^3)}^2  +  \|\na n\|_{L^2(\R^3)}^2   \nn \\
& \le  \f18\|\na n\|_{L^2(\R^3)}^2  + C\|\na n\|_{L^2(\R^3)}^{\f{6}{r_1}}  \|n\|_{L^2(\R^3)}^{\f{2(r_1-3)}{r_1}}  \|\na c\|_{L^{r_1}(\R^3)}^2
    + 4\|\na\phi\|_{L^{\infty}(\R^3)}^2\|n\|_{L^2(\R^3)}^2   \nn \\
&\le   \f14\|\na n\|_{L^2(\R^3)}^2  +  C\|n\|_{L^2(\R^3)}^2\|\na c\|_{L^{r_1}(\R^3)}^{\f{2r_1}{r_1-3}}  +  C\|n\|_{L^2(\R^3)}^2,
\end{align*}
which implies that
\begin{align}\label{blow:nL2-1}
\f{d}{dt}\|n\|_{L^2(\R^3)}^2  +  \f32\|\na n\|_{L^2(\R^3)}^2
\le  C\Big(1+\|\na c\|_{L^{r_1}(\R^3)}^{\f{2r_1}{r_1-3}} \Big)\|n\|_{L^2(\R^3)}^2.
\end{align}

Multiplying \eqref{EQ}$_2$ by $c$ and using the nonnegativity of $f$, $n$, $c$, one can directly obtain that
\begin{align}\label{blow:cL2}
\f{d}{dt}\|c\|_{L^2(\R^3)}^2   +  2\|\na c\|_{L^2(\R^3)}^2  \leq 0.
\end{align}

For the fluid velocity $\bfu$, it follows from H\"{o}lder's inequality, Young's inequality and \eqref{blow:nc} that
\begin{align}\label{blow:uL2}
&\f12\f{d}{dt}\|\bfu\|_{L^2(\R^3)}^2  +  \|\na\bfu\|_{L^2(\R^3)}^2  \nn \\
& = -\int_{\R^3}\bfu\cd n\na\phi  +  \int_{\R^3}\bfu\cd \chi(c)n\na c   \nn\\
& \le \|\bfu\|_{L^2(\R^3)}^2  +  \f12\|\na\phi\|_{L^{\infty}(\R^3)}^2\|n\|_{L^2(\R^3)}^2
    +  \f12\mathcal{C}_{\chi}^2\|n\na c\|_{L^2(\R^3)}^2   \nn\\
& \le \|\bfu\|_{L^2(\R^3)}^2  + C\|n\|_{L^2(\R^3)}^2
+  C\|\na n\|_{L^2(\R^3)}^{\f{6}{r_1}}  \|n\|_{L^2(\R^3)}^{\f{2(r_1-3)}{r_1}}  \|\na c\|_{L^{r_1}(\R^3)}^2  \nn \\
& \le \|\bfu\|_{L^2(\R^3)}^2  +  \f18\|\na n\|_{L^2(\R^3)}^2
    +  C\Big(1 +  \|\na c\|_{L^{r_1}(\R^3)}^{\f{2r_1}{r_1-3}}\Big)\|n\|_{L^2(\R^3)}^2.
\end{align}

For the estimate of $\na c$, we multiply $-\Del c$ to both sides of \eqref{EQ}$_2$ and integrate the resulting equation to get that
\begin{align*}
\f12\f{d}{dt}\|\na c\|_{L^2(\R^3)}^2  +  \|\na^2c\|_{L^2(\R^3)}^2
& = \int_{\R^3}\Del c\,\bfu\cd\na c  +  \int_{\R^3}\Del c \,nf(c)   \nn \\
& = \int_{\R^3}(\na^2c:\na\bfu)\, c  +  \int_{\R^3}\Del c \,nf(c)   \nn \\
& \le \f12\|\na^2c\|_{L^2(\R^3)}^2  +  \|\na\bfu\|_{L^2(\R^3)}^2\|c\|_{L^{\infty}(\R^3)}^2    +  \mathcal{C}_{f}^2\|n\|_{L^2(\R^3)}^2,
\end{align*}
where we used $\na\cd\bfu=0$ and the equality that
\begin{align*}
\int_{\R^3}\Del c\,\bfu\cd\na c
& = \sum_{i,j=1}^3\int_{\R^3}\pa_i\pa_i c\,u_j\,\pa_j c
 = -\sum_{i,j=1}^3\int_{\R^3}\pa_i\pa_i\pa_j c\,u_j\,c   \nn\\
& = \sum_{i,j=1}^3\int_{\R^3}\pa_i\pa_j c\,\pa_i u_j\, c  + \sum_{i,j=1}^3\int_{\R^3}\pa_i\pa_j c\, u_j\,\pa_i c
= \int_{\R^3}(\na^2c:\na\bfu)\, c.
\end{align*}
Thus we obtain
\begin{align}\label{blow:H1c}
\f{d}{dt}\|\na c\|_{L^2(\R^3)}^2  +  \|\na^2c\|_{L^2(\R^3)}^2
\le 2c_{\infty}^2\|\na\bfu\|_{L^2(\R^3)}^2   +  C\|n\|_{L^2(\R^3)}^2
\end{align}
due to $\|c\|_{L^{\infty}(\R^3)}\leq \|c_0\|_{L^{\infty}(\R^3)}=:c_{\infty}$.

Analogously, multiplying both sides of \eqref{EQ}$_3$ by $-\Del\bfu$ and using the integration by parts, one has
\begin{align}\label{blow:H1u}
& \f12\f{d}{dt}\|\na\bfu\|_{L^2(\R^3)}^2  +  \|\na^2\bfu\|_{L^2(\R^3)}^2 \nonumber \\
& = \kappa\int_{\R^3}\Del\bfu\cd(\bfu\cd\na)\bfu  +  \int_{\R^3}\Del\bfu\cd n\na\phi  -  \int_{\R^3}\Del\bfu\cd\chi(c)n\na c.
\end{align}
It follows from H\"{o}lder's inequality and Gagliardo-Nirenberg inequality that
\begin{align}\label{blow:H1u-star}
\int_{\R^3}\Del\bfu\cd(\bfu\cd\na)\bfu
&\leq \|\Del\bfu\|_{L^2(\R^3)} \|\bfu\cd\na\bfu\|_{L^2(\R^3)}   \nn \\
&\leq \|\Del\bfu\|_{L^2(\R^3)}\|\bfu\|_{L^{r_2}(\R^3)}\|\na\bfu\|_{L^{\f{2r_2}{r_2-2}}(\R^3)}  \nn \\
&\leq C\|\na^2\bfu\|_{L^2(\R^3)}^{\f{r_2+3}{r_2}}\|\bfu\|_{L^{r_2}(\R^3)}\|\na\bfu\|_{L^2(\R^3)}^{\f{r_2-3}{r_2}}.
\end{align}
Thus by applying Young's inequality and using \eqref{blow:nc}, we have
\begin{align*}
&\f12\f{d}{dt}\|\na\bfu\|_{L^2(\R^3)}^2  +  \|\na^2\bfu\|_{L^2(\R^3)}^2  \nn \\
&\le C\kappa\|\na^2\bfu\|_{L^2(\R^3)}^{\f{r_2+3}{r_2}}\|\bfu\|_{L^{r_2}(\R^3)}\|\na\bfu\|_{L^2(\R^3)}^{\f{r_2-3}{r_2}}
  +  \|\Del\bfu\|_{L^2(\R^3)}\|n\|_{L^2(\R^3)}\|\na\phi\|_{L^{\infty}(\R^3)}  \nn \\
&\qquad  + C\|\Del\bfu\|_{L^2(\R^3)}\|\na n\|_{L^2(\R^3)}^{\f{3}{r_1}}  \|n\|_{L^2(\R^3)}^{\f{r_1-3}{r_1}}  \|\na c\|_{L^{r_1}(\R^3)}  \nn \\
&\le  \f12\|\na^2\bfu\|_{L^2(\R^3)}^2  +  C\kappa^{\f{2r_2}{r_2-3}}\|\bfu\|_{L^{r_2}(\R^3)}^{\f{2r_2}{r_2-3}}\|\na\bfu\|_{L^2(\R^3)}^2  \nn \\
&\qquad  +  C\|n\|_{L^2(\R^3)}^2  +  \f18\|\na n\|_{L^2(\R^3)}^2
   +  C\|n\|_{L^2(\R^3)}^2\|\na c\|_{L^{r_1}(\R^3)}^{\f{2r_1}{r_1-3}},
\end{align*}
which implies that
\begin{align}\label{blow:H1u-1}
&\f{d}{dt}\|\na\bfu\|_{L^2(\R^3)}^2  +  \|\na^2\bfu\|_{L^2(\R^3)}^2\nn \\
&\le C\kappa^{\f{2r_2}{r_2-3}}\|\bfu\|_{L^{r_2}(\R^3)}^{\f{2r_2}{r_2-3}}\|\na\bfu\|_{L^2(\R^3)}^2   +  \f14\|\na n\|_{L^2(\R^3)}^2
   +  C\Big(1+\|\na c\|_{L^{r_1}(\R^3)}^{\f{2r_1}{r_1-3}} \Big)\|n\|_{L^2(\R^3)}^2.
\end{align}
Collecting \eqref{blow:nL2-1}-\eqref{blow:H1c} and \eqref{blow:H1u-1}, we have
\begin{align*}
&\f{d}{dt} \Big(\|n\|_{L^2(\R^3)}^2+\|c\|_{H^1(\R^3)}^2 + \|\bfu\|_{H^1(\R^3)}^2\Big)
   + \Big(\|\na n\|_{L^2(\R^3)}^2+\|\na c\|_{H^1(\R^3)}^2 + \|\na\bfu\|_{H^1(\R^3)}^2\Big)  \nn \\
& \le C\Big(1+\|\na c\|_{L^{r_1}(\R^3)}^{\f{2r_1}{r_1-3}} +  \kappa^{\f{2r_2}{r_2-3}}\|\bfu\|_{L^{r_2}(\R^3)}^{\f{2r_2}{r_2-3}}\Big)\Big(\|n\|_{L^2(\R^3)}^2+\|\bfu\|_{H^1(\R^3)}^2 \Big)
\end{align*}
and thus deduce from Gronwall's inequality that
\begin{align}\label{blow:L2ncu-1}
&\sup_{0\leq t\leq T^*} \Big(\|n\|_{L^2(\R^3)}^2+\|c\|_{H^1(\R^3)}^2 + \|\bfu\|_{H^1(\R^3)}^2 \Big) \nn \\
 &\qquad +   \int_0^{T^*}\lt(\|\na n\|_{L^2(\R^3)}^2+\|\na c\|_{H^1(\R^3)}^2+\|\na\bfu\|_{H^1(\R^3)}^2\rt) dt  \nn \\
&\le C\Big(\|n_0\|_{L^2(\R^3)}^2+\|c_0\|_{H^1(\R^3)}^2+\|\bfu_0\|_{H^1(\R^3)}^2\Big)   \nn \\
&\qquad   \cdot \exp\Big\{\int_0^{T^*} \Big(1 +  \|\na c\|_{L^{r_1}(\R^3)}^{s_1} +  \kappa^{\f{2r_2}{r_2-3}}\|\bfu\|_{L^{r_2}(\R^3)}^{s_2}\Big) dt \Big\}
\leq C_1,
\end{align}
where $C_1$ is a positive constant depending only on initial data and $T^*$, and we also used the assumptions \eqref{blow:no-1}-\eqref{blow:no-2} and the facts that $\f{2r_1}{r_1-3}\leq s_1$ and $\f{2r_2}{r_2-3}\leq s_2$.

\vspace{2mm}
\textbf{Step 2}: Estimate of $(\na n,\na^2c,\na^2\bfu)$ in $L^{\infty}_tL_x^2\times L^{\infty}_tL_x^2\times L^{\infty}_tL_x^2$.

Considering that the above obtained estimate \eqref{blow:L2ncu-1} is not enough to control the strong nonlinear chemotaxis term $\na\cd(\chi(c)n\na c)$ when $\chi(c)$ is not constant, we estimate $\na^2c$ firstly. To this end, applying $\na^2$ to both side of \eqref{EQ}$_2$, multiplying the resulting equation by $\na^2c$, and using the integration by parts, we have
\begin{align}\label{blow:cH2-1}
\f12\f{d}{dt}\|\na^2c\|_{L^2(\R^3)}^2   +   \|\na^3c\|_{L^2(\R^3)}^2
&=-\int_{\R^3}\na^2c\cd\na^2(\bfu\cd\na c)  -  \int_{\R^3}\na^2c\cd\na^2\big(nf(c)\big)   \nn \\
&=\int_{\R^3}\na\Del c\cd\na(\bfu\cd\na c)  +  \int_{\R^3}\na\Del c\cd\na\big(nf(c)\big).
\end{align}
By H\"{o}lder's inequality and Gagliardo-Nirenberg inequality, one has
\begin{align}\label{I-1-1}
\int_{\R^3}\na\Del c\cd\na(\bfu\cd\na c)
&\leq \|\na^3c\|_{L^2(\R^3)}\|\na(\bfu\cd\na c)\|_{L^2(\R^3)}  \nn \\
&\leq \|\na^3c\|_{L^2(\R^3)}\Big(\|\na\bfu\|_{L^3(\R^3)}\|\na c\|_{L^6(\R^3)}
   +  \|\bfu\|_{L^{\infty}(\R^3)}\|\na^2c\|_{L^2(\R^3)}\Big)   \nn\\
&\leq \f14\|\na^3c\|_{L^2(\R^3)}^2  +  C\|\bfu\|_{H^2(\R^3)}^2\|\na^2c\|_{L^2(\R^3)}^2.
\end{align}
For the second term on the right-hand-side of \eqref{blow:cH2-1}, we deduce from H\"{o}lder's inequality and \eqref{blow:nc} that
\begin{align}\label{I-1-2}
\int_{\R^3}\na\Del c\cd\na\big(nf(c)\big)
&\le \|\na^3c\|_{L^2(\R^3)}\|\na\big(nf(c)\big)\|_{L^2(\R^3)}   \nn\\
&\le \f14\|\na^3c\|_{L^2(\R^3)}^2  +  2\mathcal{C}_{f}^2\|\na n\|_{L^2(\R^3)}^2   +  2\mathcal{C}_{f}^2\|n\na c\|_{L^2(\R^3)}^2  \nn \\
&\le \f14\|\na^3c\|_{L^2(\R^3)}^2  +  2\mathcal{C}_{f}^2\|\na n\|_{L^2(\R^3)}^2
    + C\|\na n\|_{L^2(\R^3)}^{\f{6}{r_1}}  \|n\|_{L^2(\R^3)}^{\f{2(r_1-3)}{r_1}}  \|\na c\|_{L^{r_1}(\R^3)}^2  \nn \\
&\leq \f14\|\na^3c\|_{L^2(\R^3)}^2  +  C\|\na n\|_{L^2(\R^3)}^2   +  C\|n\|_{L^2(\R^3)}^2\|\na c\|_{L^{r_1}(\R^3)}^{\f{2r_1}{r_1-3}}.
\end{align}
Substituting \eqref{I-1-1} and \eqref{I-1-2} into \eqref{blow:cH2-1}, we obtain
\begin{align*}
&\f{d}{dt}\|\na^2c\|_{L^2(\R^3)}^2   +   \|\na^3c\|_{L^2(\R^3)}^2   \nn \\
& \le   C \|\bfu\|_{H^2(\R^3)}^2\|\na^2c\|_{L^2(\R^3)}^2  +  C\|\na n\|_{L^2(\R^3)}^2
   +  C\|n\|_{L^2(\R^3)}^2\|\na c\|_{L^{r_1}(\R^3)}^{\f{2r_1}{r_1-3}},
\end{align*}
which implied by Gronwall's inequality, \eqref{blow:no-1}-\eqref{blow:no-2}, and \eqref{blow:L2ncu-1} that
\begin{align}\label{blow:cH2}
&\sup_{0\leq t\leq T^*}\|\na^2 c\|_{L^2(\R^3)}^2  +   \int_0^{T^*}\|\na^3 c\|_{L^2(\R^3)}^2 dt   \nn \\
&\le C\exp\Big\{\int_0^{T^*}\|\bfu\|_{H^2(\R^3)}^2 dt \Big\}
   \Big(\|\na^2 c_0\|_{L^2(\R^3)}^2 + \int_0^{T^*}\|\na n\|_{L^2(\R^3)}^2 dt   \nn \\
&\qquad  + \sup_{0\leq t\leq T^*}\|n\|_{L^2(\R^3)}^2
    \int_0^{T^*}\|\na c\|_{L^{r_1}(\R^3)}^{s_1}dt\Big)
\leq C_2,
\end{align}
where $C_2$ is a positive constant depending only on initial data and $T^*$.

With this at hand, we can now estimate $\na n$. Multiplying both side of \eqref{EQ}$_1$ by $-\Del n$ and integrating, we have
\begin{align}\label{blow:H1n-1}
\f12\f{d}{dt}\|\na n\|_{L^2(\R^3)}^2  +  \|\na^2 n\|_{L^2(\R^3)}^2
&=\int_{\R^3}\Del n\big(\bfu\cd\na n  +  \na\cd\lt(\chi(c)n\na c\rt)  -   \na\cd(n\na\phi)\big)  \nn \\
&\leq \f12\|\Del n\|_{L^2(\R^3)}^2   +  C\|\bfu\|_{L^{\infty}(\R^3)}^2\|\na n\|_{L^2(\R^3)}^2   \nn \\
&\qquad  +C\|\na\big(\chi(c)n\na c\big)\|_{L^2(\R^3)}^2  +  C\|\na(n\na\phi)\|_{L^2(\R^3)}^2.
\end{align}
For the last two terms on the right-hand-side of \eqref{blow:H1n-1}, it follows from H\"{o}lder's inequality and Gagliardo-Nirenberg inequality that
\begin{align}\label{blow:H1n-1-1}
\|\na\big(\chi(c)n\na c\big)\|_{L^2(\R^3)}^2
&\leq 3\mathcal{C}_{\chi}^2\|n\|_{L^6(\R^3)}^2\|\na c\|_{L^6(\R^3)}^4  +  3\mathcal{C}_{\chi}^2\|\na n\|_{L^2(\R^3)}^2\|\na c\|_{L^{\infty}(\R^3)}^2   \nn \\
&\qquad  +  3\mathcal{C}_{\chi}^2\|n\|_{L^6(\R^3)}^2\|\na^2 c\|_{L^3(\R^3)}^2  \nn \\
&\leq C\|\na n\|_{L^2(\R^3)}^2\|\na^2 c\|_{L^2(\R^3)}^4
   +  C\|\na n\|_{L^2(\R^3)}^2\|\na c\|_{H^2(\R^3)}^2
\end{align}
and
\begin{align}\label{blow:H1n-1-2}
\|\na(n\na\phi)\|_{L^2(\R^3)}^2
\leq 2\|\na\phi\|_{L^{\infty}(\R^3)}^2 \|\na n\|_{L^2(\R^3)}^2  +  2\|\na^2\phi\|_{L^{\infty}(\R^3)}^2 \|n\|_{L^2(\R^3)}^2.
\end{align}
Substituting \eqref{blow:H1n-1-1} and \eqref{blow:H1n-1-2} into \eqref{blow:H1n-1}, we conclude that
\begin{align*}
&\f{d}{dt}\|\na n\|_{L^2(\R^3)}^2  +  \|\na^2 n\|_{L^2(\R^3)}^2 \nn \\
&\leq C\Big(1  +  \|\bfu\|_{H^2(\R^3)}^2  + \|\na^2 c\|_{L^2(\R^3)}^4
+\|\na c\|_{H^2(\R^3)}^2 \Big)
  \|\na n\|_{L^2(\R^3)}^2   +  C\|n\|_{L^2(\R^3)}^2,
\end{align*}
which implied by Gronwall's inequality, \eqref{blow:L2ncu-1} and \eqref{blow:cH2} that
\begin{align}\label{blow:H1n}
&\sup_{0\leq t\leq T^*}\|\na n\|_{L^2(\R^3)}^2  +   \int_0^{T^*}\|\na^2 n\|_{L^2(\R^3)}^2 dt   \nn \\
&\le C\Big(\|\na n_0\|_{L^2(\R^3)}^2  + T^*\sup_{0\leq t\leq T^*}\|n\|_{L^2(\R^3)}^2\Big)  \nn \\
&\qquad  \cdot \exp\Big\{\int_0^{T^*} \Big(1  +  \|\bfu\|_{H^2(\R^3)}^2  + \|\na^2 c\|_{L^2(\R^3)}^4
    +\|\na c\|_{H^2(\R^3)}^2 \Big) dt \Big\}
\leq C_3,
\end{align}
where $C_3$ is a positive constant depending only on initial data and $T^*$.

For the fluid velocity $\bfu$, using \eqref{blow:H1n-1-1} and \eqref{blow:H1n-1-2}, we can derive that
\begin{align*}
&\f12\f{d}{dt}\|\na^2\bfu\|_{L^2(\R^3)}^2  +  \|\na^3\bfu\|_{L^2(\R^3)}^2   \nn\\
&=\int_{\R^3}\na^2\bfu\cd\na^2\big(-\kappa(\bfu\cd\na)\bfu  -n\na\phi   +   \chi(c)n\na c\big)   \nn \\
&=-\int_{\R^3}\na\Del\bfu\cd\na\big(-\kappa(\bfu\cd\na)\bfu  -n\na\phi   +   \chi(c)n\na c\big)   \nn \\
&\leq \f12\|\na^3\bfu\|_{L^2(\R^3)}^2  + C\kappa^2\|\na\bfu\|_{L^3(\R^3)}^2\|\na\bfu\|_{L^6(\R^3)}^2
  +  C\kappa^2\|\bfu\|_{L^{\infty}(\R^3)}^2\|\na^2\bfu\|_{L^2(\R^3)}^2   \nn\\
&\quad  +  C\|\na(n\na\phi)\|_{L^2(\R^3)}^2  +  C\|\na(\chi(c)n\na c)\|_{L^2(\R^3)}^2  \nn \\
&\leq  \f12\|\na^3\bfu\|_{L^2(\R^3)}^2   +  C\kappa^2\|\bfu\|_{H^2(\R^3)}^2\|\na^2\bfu\|_{L^2(\R^3)}^2   + C\|n\|_{H^1(\R^3)}^2  \nn\\
&\quad +  C\|\na n\|_{L^2(\R^3)}^2\lt(\|\na^2 c\|_{L^2(\R^3)}^4
     +  \|\na c\|_{H^2(\R^3)}^2\rt),
\end{align*}
which implied by Gronwall's inequality, \eqref{blow:L2ncu-1}, \eqref{blow:cH2} and \eqref{blow:H1n} that
\begin{align*}
&\sup_{0\leq t\leq T^*}\|\na^2\bfu\|_{L^2(\R^3)}^2  +   \int_0^{T^*}\|\na^3\bfu\|_{L^2(\R^3)}^2 dt   \nn \\
&\leq C\exp\lt\{\kappa^2\int_0^{T^*}\|\bfu\|_{H^2(\R^3)}^2 dt\rt\}
   \Bigg(\|\na^2\bfu_0\|_{L^2(\R^3)}^2  +\int_0^{T^*}\|n\|_{H^1(\R^3)}^2 dt  \nn \\
&\quad + \sup_{0\leq t\leq T^*}\|\na n\|_{L^2(\R^3)}^2  \bigg(T^*\sup_{0\leq t\leq T^*}\|\na^2c\|_{L^2(\R^3)}^4
   + \int_0^{T^*} \|\na c\|_{H^2(\R^3)}^2 dt\bigg)\Bigg)
\leq C_4,
\end{align*}
where $C_4$ is a positive constant depending only on initial data and $T^*$.

\vspace{2mm}

\textbf{Step 3}: Estimate of $(n,c,\bfu)$ in $L^{\infty}_tH_x^{m-1}\times L^{\infty}_tH_x^m\times L^{\infty}_tH_x^m$ $(m\geq3)$.

Now, we are ready to estimate $(n,c,\bfu)$ in $H^{m-1}\times H^m\times H^m$ space for $m\geq1$ by induction. From the above two steps, the case $m=1,2$ are proved. To deal with the case $m\geq3$, we firstly take $\pa^{\alpha}$ $(2\leq|\alpha|\leq m-1)$ derivative to both side of equation \eqref{EQ}$_1$, multiply the resulting equation by $\pa^{\alpha}n$, and integrate to get that
\begin{align}\label{Hmn-1}
\f12\f{d}{dt}\|\pa^{\alpha}n\|_{L^2(\R^3)}^2  +  \|\na\pa^{\alpha}n\|_{L^2(\R^3)}^2
&=-\int_{\R^3}\pa^{\alpha}n\,\pa^{\alpha}(\bfu\cd\na n)   -  \int_{\R^3}\pa^{\alpha}n\,\pa^{\alpha}\big(\na\cd(\chi(c)n\na c)\big)  \nn \\
&  \quad\,\, +\int_{\R^3}\pa^{\alpha}n\,\pa^{\alpha}\big(\na\cd(n\na\phi)\big).
\end{align}
We will estimate the three terms on the right-hand-side of \eqref{Hmn-1} one by one. Firstly, by divergence-free property of $\bfu$ and the tame estimates in Lemma \ref{leibniz}, we have
\begin{align}\label{I-5-1}
-\int_{\R^3}\pa^{\alpha}n\,\pa^{\alpha}(\bfu\cd\na n)
 &=\int_{\R^3}\na\pa^{\alpha}n\cd\pa^{\alpha}(\bfu n)  \nn \\
&\leq C\|\na n\|_{H^{m-1}(\R^3)}\lt(\|\bfu\|_{H^{m-1}(\R^3)}\|n\|_{L^{\infty}(\R^3)}
   + \|\bfu\|_{L^{\infty}(\R^3)}\|n\|_{H^{m-1}(\R^3)}\rt) \nn\\
&\leq C\|\na n\|_{H^{m-1}(\R^3)}\|\bfu\|_{H^{m-1}(\R^3)}\|n\|_{H^{m-1}(\R^3)}
\end{align}
due to $m\geq3$. Similarly, we proceed to have
\begin{align}\label{I-6-1}
-  \int_{\R^3}\pa^{\alpha}n\,\pa^{\alpha}\big(\na\cd(\chi(c)n\na c)\big)
&= \int_{\R^3}\na\pa^{\alpha}n\cd\pa^{\alpha}\big(\chi(c)n\na c\big)  \nn \\
&\leq C\|\na n\|_{H^{m-1}(\R^3)}\|\chi(c)n\na c\|_{H^{m-1}(\R^3)}   \nn \\
&\leq C\|\na n\|_{H^{m-1}(\R^3)}\|n\|_{H^{m-1}(\R^3)}\|\chi(c)\na c\|_{H^{m-1}(\R^3)}.
\end{align}
For the third factor on the right-hand-side of \eqref{I-6-1}, we use the Leibniz formula, Sobolev embedding, and Moser estimate in Lemma \ref{moser}  to get that
\begin{align}\label{I-6-3}
\|\chi(c)\na c\|_{H^{m-1}(\R^3)}
&\leq C\|c\|_{H^m(\R^3)}  +  \sum_{1\leq|\beta|\leq m-1}\bigg\|\sum_{\ga\leq\beta,|\ga|<|\beta|}C^{\gamma}_{\beta}\pa^{\beta-\gamma}\chi(c)\na\pa^{\gamma}c\bigg\|_{L^2(\R^3)}  \nn \\
&\leq C\|c\|_{H^m(\R^3)}  +  C\sum_{1\leq|\beta|\leq m-1}\bigg(\|\pa^{\beta}\chi(c)\|_{L^2(\R^3)}\|\na c\|_{L^{\infty}(\R^3)}  \nn \\
&\qquad    + \sum_{|\gamma|=1}\|\pa^{\beta-\gamma}\chi(c)\|_{L^4(\R^3)}\|\na\pa^{\gamma}c\|_{L^4(\R^3)}  +  \cdots  \nn\\
&\qquad +  \sum_{|\gamma|=|\beta|-1}\|\pa^{\beta-\gamma}\chi(c)\|_{L^{\infty}(\R^3)}\|\na\pa^{\gamma}c\|_{L^2(\R^3)}\bigg)  \nn \\
&\leq C\|c\|_{H^m(\R^3)}  + C\sum_{1\leq|\beta|\leq m-1}
     \bigg(\|c\|_{L^{\infty}(\R^3)}^{|\beta|-1}\|\na^{|\beta|}c\|_{L^2(\R^3)}\|\na c\|_{H^2(\R^3)} \nn \\
&\qquad +  \|c\|_{L^{\infty}(\R^3)}^{|\beta|-2}\|\na^{|\beta|-1}c\|_{L^4(\R^3)}\|\na^2 c\|_{L^4(\R^3)}  +  \cdots  \nn\\
&\qquad + \|\na c\|_{L^{\infty}(\R^3)}\|\na^{|\beta|}c\|_{L^2(\R^3)}\bigg)   \nn \\
&\leq C\|c\|_{H^m(\R^3)}  +  C\lt(1+\|c\|_{L^{\infty}(\R^3)}^{m-2}\rt)\|c\|_{H^{m-1}(\R^3)}\|\na c\|_{H^{m-1}(\R^3)}  \nn \\
&\leq C\|c\|_{H^m(\R^3)}  +  C\lt(1+c_{\infty}^{m-2}\rt)\|c\|_{H^{m-1}(\R^3)}\|\na c\|_{H^{m-1}(\R^3)},
\end{align}
which together with \eqref{I-6-1} yields that
\begin{align}\label{I-6-2}
&-  \int_{\R^3}\pa^{\alpha}n\,\pa^{\alpha}\big(\na\cd(\chi(c)n\na c)\big)  \nn \\
&\le C\|\na n\|_{H^{m-1}(\R^3)}\|n\|_{H^{m-1}(\R^3)}\lt(\|c\|_{H^m(\R^3)}  +  \|c\|_{H^{m-1}(\R^3)}\|\na c\|_{H^{m-1}(\R^3)}\rt).
\end{align}
Finally, due to the assumption $\phi\in W^{m,\infty}(\R^3)$, we can deduce from Leibniz formula that
\begin{align}\label{I-7-1}
\int_{\R^3}\pa^{\alpha}n\,\pa^{\alpha}\big(\na\cd(n\na\phi)\big)
&=-\int_{\R^3}\na\pa^{\alpha}n\cd\pa^{\alpha}(n\na\phi)   \nn \\
&\leq C\|\na n\|_{H^{m-1}(\R^3)}\|n\na\phi\|_{H^{m-1}(\R^3)}   \nn \\
&\leq C\|\na n\|_{H^{m-1}(\R^3)}\|n\|_{H^{m-1}(\R^3)}.
\end{align}
Then after substituting \eqref{I-5-1}, \eqref{I-6-2} and \eqref{I-7-1} into \eqref{Hmn-1} and taking summation over $|\alpha|\leq m-1$, it shows
\begin{align}\label{Hmn}
&\f{d}{dt}\|n\|_{H^{m-1}(\R^3)}^2  +  \|\na n\|_{H^{m-1}(\R^3)}^2   \nn \\
&\leq C\Big(1+ \|\bfu\|_{H^{m-1}(\R^3)}^2 + \|c\|_{H^m(\R^3)}^2  + \|c\|_{H^{m-1}(\R^3)}^2\|\na c\|_{H^{m-1}(\R^3)}^2\Big)\|n\|_{H^{m-1}(\R^3)}^2 .
\end{align}
Similarly, by the divergence-free property of $\bfu$, for $3\leq|\alpha|\leq m$, we deduce that
\begin{align*}
&\f12\f{d}{dt}\|\pa^{\alpha}c\|_{L^2(\R^3)}^2  +  \|\na\pa^{\alpha}c\|_{L^2(\R^3)}^2  \nn \\
&=-\int_{\R^3}\pa^{\alpha}c\,\pa^{\alpha}(\bfu\cd\na c)  -  \int_{\R^3}\pa^{\alpha}c\,\pa^{\alpha}\big(nf(c)\big)  \nn\\
&=-\int_{\R^3}\pa^{\alpha}c\,\big(\pa^{\alpha}(\bfu\cd\na c)-(\bfu\cd\na)\pa^{\alpha}c\big)  -  \int_{\R^3}\pa^{\alpha}c\,\pa^{\alpha}\big(nf(c)\big)  \nn\\
&\leq C\|c\|_{H^m(\R^3)}\Big(\|\bfu\|_{H^m(\R^3)}\|\na c\|_{L^{\infty}(\R^3)}   +  \|\na\bfu\|_{L^{\infty}(\R^3)}\|c\|_{H^m(\R^3)}\Big)  \nn\\
&\qquad  + C\|\na c\|_{H^m(\R^3)}\|nf(c)\|_{H^{m-1}(\R^3)},
\end{align*}
where
\begin{align*}
\|nf(c)\|_{H^{m-1}(\R^3)}
&\leq C\|n\|_{H^{m-1}(\R^3)}
   + \sum_{1\leq|\beta|\leq m-1}\bigg\|\sum_{\ga\leq\beta,|\ga| < |\beta|}C_{\beta}^{\ga}\pa^{\ga}n\,\pa^{\beta-\ga}f(c)\bigg\|_{L^2(\R^3)}   \\
&\leq  C\|n\|_{H^{m-1}(\R^3)}
   + C\sum_{1\leq|\beta|\leq m-1}\bigg(\|n\|_{L^{\infty}(\R^3)}\|\pa^{\beta}f(c)\|_{L^2(\R^3)}  \\
&\qquad   +  \sum_{|\ga|=1}\|\pa^{\ga}n\|_{L^4(\R^3)}\|\pa^{\beta-\ga}f(c)\|_{L^4(\R^3)}  +  \cdots  \nn\\
&\qquad + \sum_{|\ga|=|\beta|-1}\|\pa^{\ga}n\|_{L^2(\R^3)}\|\pa^{\beta-\ga}f(c)\|_{L^{\infty}(\R^3)}\bigg)  \nn \\
&\leq  C\|n\|_{H^{m-1}(\R^3)}
   + C\sum_{1\leq|\beta|\leq m-1}\bigg(\|n\|_{L^{\infty}(\R^3)}\|c\|_{L^{\infty}(\R^3)}^{|\beta|-1}\|\na^{|\beta|}c\|_{L^2(\R^3)}  \nn \\
&\qquad +  \|\na n\|_{L^4(\R^3)}\|c\|_{L^{\infty}(\R^3)}^{|\beta|-2}\|\na^{|\beta|-1}c\|_{L^4(\R^3)} +  \cdots  \nn\\
&\qquad + \|\na^{|\beta|-1}n\|_{L^2(\R^3)}\|\na c\|_{L^{\infty}(\R^3)}\bigg)  \nn\\
&\leq  C\|n\|_{H^{m-1}(\R^3)}
   + C \lt(1+\|c\|_{L^{\infty}(\R^3)}^{m-2}\rt)\|n\|_{H^{m-1}(\R^3)}\|c\|_{H^{m}(\R^3)}  \nn \\
&\leq  C\|n\|_{H^{m-1}(\R^3)}
   + C \lt(1+c_{\infty}^{m-2}\rt)\|n\|_{H^{m-1}(\R^3)}\|c\|_{H^{m}(\R^3)}.
\end{align*}
Thus we have
\begin{align}\label{Hmc}
&\f{d}{dt}\|c\|_{H^{m}(\R^3)}^2  +  \|\na c\|_{H^{m}(\R^3)}^2  \nn \\
&\leq C\Big(1+ \|\na c\|_{H^{m-1}(\R^3)}  +  \|c\|_{H^{m}(\R^3)}^2 + \|\na\bfu\|_{H^{m-1}(\R^3)}\Big)  \nn \\
&\qquad \cdot \Big(\|n\|_{H^{m-1}(\R^3)}^2 + \|c\|_{H^{m}(\R^3)}^2 + \|\bfu\|_{H^{m}(\R^3)}^2 \Big)   \nn \\
&\le C\Big(1+ \|c\|_{H^{m}(\R^3)}^2 + \|\bfu\|_{H^{m}(\R^3)}^2\Big)  \Big(\|n\|_{H^{m-1}(\R^3)}^2 + \|c\|_{H^{m}(\R^3)}^2 + \|\bfu\|_{H^{m}(\R^3)}^2\Big).
\end{align}
Finally, for the $H^m$ estimate of $\bfu$, one has
\begin{align*}
&\f12\f{d}{dt}\|\pa^{\alpha}\bfu\|_{L^2(\R^3)}^2  +  \|\na\pa^{\alpha}\bfu\|_{L^2(\R^3)}^2  \nn \\
&=-\kappa\int_{\R^3}\pa^{\alpha}\bfu\cd\Big(\pa^{\alpha}\big((\bfu\cd\na)\bfu\big)-(\bfu\cd\na)\pa^{\alpha}\bfu\Big)
   -  \int_{\R^3}\pa^{\alpha}\bfu\cd\Big(\pa^{\alpha}(n\na\phi)  -  \pa^{\alpha}\big(\chi(c)n\na c\big)\Big)\nn\\
&\leq C\kappa\|\na\bfu\|_{L^{\infty}(\R^3)}\|\bfu\|_{H^m(\R^3)}^2 + C\|\na\bfu\|_{H^{m}(\R^3)}\Big(\|n\na\phi\|_{H^{m-1}(\R^3)} + \|\chi(c)n\na c\|_{H^{m-1}(\R^3)}\Big)
\end{align*}
for $3\leq|\alpha|\leq m$, which together with \eqref{I-6-3} and the summation over $|\alpha|\leq m$ shows that
\begin{align}\label{Hmu}
\f{d}{dt}\|\bfu\|_{H^{m}(\R^3)}^2  +  \|\na\bfu\|_{H^{m}(\R^3)}^2
&\leq  C\kappa\|\na\bfu\|_{H^{m-1}(\R^3)}\|\bfu\|_{H^{m}(\R^3)}^2
 + C\Big(1 + \|c\|_{H^m(\R^3)}^2   \nn \\
&\qquad + \|c\|_{H^{m-1}(\R^3)}^2\|\na c\|_{H^{m-1}(\R^3)}^2\Big)\|n\|_{H^{m-1}(\R^3)}^2.
\end{align}
Collecting \eqref{Hmn}, \eqref{Hmc}, \eqref{Hmu} and the estimates in Step 1-Step 2, we obtain
\begin{align*}
&\f{d}{dt}\Big(\|n\|_{H^{m-1}(\R^3)}^2 +\|c\|_{H^m(\R^3)}^2+\|\bfu\|_{H^m(\R^3)}^2\Big)  \nn \\
&\qquad  +  \Big(\|\na n\|_{H^{m-1}(\R^3)}^2 +\|\na c\|_{H^m(\R^3)}^2+\|\na\bfu\|_{H^m(\R^3)}^2\Big)   \nn \\
&\leq C\Big(1+ \|c\|_{H^m(\R^3)}^2 + \|\bfu\|_{H^m(\R^3)}^2 +  \|c\|_{H^{m-1}(\R^3)}^2\|\na c\|_{H^{m-1}(\R^3)}^2\Big)  \nn \\
&\qquad \cdot \Big(\|n\|_{H^{m-1}(\R^3)}^2 +\|c\|_{H^m(\R^3)}^2+\|\bfu\|_{H^m(\R^3)}^2\Big).
\end{align*}
Using Gronwall's inequality, we conclude that if \eqref{blow:no-1}-\eqref{blow:no-2} is true, it holds
\begin{align*}
&\sup_{0\leq t\leq T^*}\Big(\|n\|_{H^{m-1}(\R^3)}^2 + \|c\|_{H^m(\R^3)}^2+\|\bfu\|_{H^m(\R^3)}^2\Big) \nn\\
&\qquad +   \int_0^{T^*}\Big(\|\na n\|_{H^{m-1}(\R^3)}^2 + \|\na c\|_{H^m(\R^3)}^2 + \|\na\bfu\|_{H^m(\R^3)}^2\Big) dt  \nn \\
&\leq C\Big( \|n_0\|_{H^{m-1}(\R^3)}^2 + \|c_0\|_{H^m(\R^3)}^2 + \|\bfu_0\|_{H^m(\R^3)}^2 \Big)  \nn \\
&\qquad \cdot \exp\Big\{\int_0^{T^*}\Big(1+\|c\|_{H^m(\R^3)}^2 + \|\bfu\|_{H^m(\R^3)}^2\Big)dt  \nn \\
 &\qquad + \sup_{0\leq t\leq T^*}\|c\|_{H^{m-1}(\R^3)}^2\int_0^{T^*}\|\na c\|_{H^{m-1}(\R^3)}^2 dt \Big\}
 \leq C_5,
\end{align*}
where $C_5$ is a positive constant depending only on initial data, $m$ and $T^*$. This contradicts the assumption that $T^*$ is the maximal time of existence with $T^*<\infty$. Thus we have completed the proof of Theorem \ref{blowup1}.
\end{proof}

\subsection{Extensibility criterion in $\R^2$}

The proof of Theorem \ref{blowup2} is similar to that of Theorem \ref{blowup1}. We give a sketch for the completeness.
\begin{proof}[Proof of Theorem \ref{blowup2}]
For any given $\kappa\in\R$, suppose conversely that
\begin{align}\label{blow:no-3}
\int_0^{T^*}\|\na c\|_{L^{r_3}(\R^2)}^{s_3} dt < \infty
\end{align}
for some $(r_3,s_3)$ satisfying $\f{2}{r_3}+\f{2}{s_3}\leq 1$ and  $2<r_3\leq \infty$.

\vspace{2mm}

\textbf{Step 1}: Estimate of $(n,c,\bfu)$ in $L^{\infty}_tL_x^2\times L^{\infty}_tH_x^1\times L^{\infty}_tH_x^1$.

We first estimate the two dimensional \eqref{blow:nc} as
\begin{align}\label{blow:nc-1}
\|n\na c\|_{L^2(\R^2)}^2
\leq \|n\|_{L^{\f{2r_3}{r_3-2}}(\R^2)}^2\|\na c\|_{L^{r_3}(\R^2)}^2
\leq C\|\na n\|_{L^2(\R^2)}^{\f{4}{r_3}}\|n\|_{L^2(\R^2)}^{\f{2(r_3-2)}{r_3}}\|\na c\|_{L^{r_3}(\R^2)}^2
\end{align}
and then  obtain
\begin{equation}\label{blow:nL2-2}
\f{d}{dt}\|n\|_{L^2(\R^2)}^2 + \|\na n\|_{L^2(\R^2)}^2
\le C\Big(1+ \|\na c\|_{L^{r_3}(\R^2)}^{\f{2r_3}{r_3-2}}\Big) \|n\|_{L^2(\R^2)}^2
\end{equation}
and
\begin{equation}\label{blow:L2u}
\f12\f{d}{dt}\|\bfu\|_{L^2(\R^2)}^2  +  \|\na\bfu\|_{L^2(\R^2)}^2
\le \|\bfu\|_{L^2(\R^2)}^2  +  \f12\|\na n\|_{L^2(\R^2)}^2
 +  C\Big(1 + \|\na c\|_{L^{r_3}(\R^2)}^{\f{2r_3}{r_3-2}}\Big) \|n\|_{L^2(\R^2)}^2
\end{equation}
by following the proof of \eqref{blow:nL2-1} and \eqref{blow:uL2}, respectively.  It is clear that \eqref{blow:cL2} still holds in the two dimensional setting. Then we have from \eqref{blow:nL2-2} and \eqref{blow:L2u} that
\begin{align}\label{blow:L2nu}
&\sup_{0\le t\le T^*} \big\|(n, c, \bfu) \big\|_{L^2(\R^2)}^2
    +   \int_0^{T^*}\big\|\big(\na n, \na c, \na\bfu\big)\big\|_{L^2(\R^2)}^2  dt  \nn \\
&\le C \big\|\big(n_0, c_0, \bfu_0\big)\big\|_{L^2(\R^2)}^2  \exp\Big\{\int_0^{T^*} \Big(1 +  \|\na c\|_{L^{r_3}(\R^2)}^{s_3}\Big) dt\Big\}
\le C_1,
\end{align}
due to $\f{2r_3}{r_3-3}\leq s_3$ by applying Gronwall's inequality and \eqref{blow:no-3}, where $C_1$ is a positive constant depending only on initial data and $T^*$.

The estimate of $\na c$ is totally same as in \eqref{blow:H1c}, i.e. we have
\begin{align}\label{blow:H1c-1}
\f{d}{dt}\|\na c\|_{L^2(\R^2)}^2  +  \|\na^2c\|_{L^2(\R^2)}^2
\leq 2c_{\infty}^2\|\na\bfu\|_{L^2(\R^2)}^2   +  C\|n\|_{L^2(\R^2)}^2.
\end{align}
On the other hand,  to deal with \eqref{blow:H1u} in the two dimensional setting, we estimate \eqref{blow:H1u-star} by using Gagliardo-Nirenberg inequality as
\[
\int_{\R^2}\Del\bfu\cd(\bfu\cd\na)\bfu
\le \|\Del\bfu\|_{L^2(\R^2)}\|\bfu\|_{L^4(\R^2)}\|\na\bfu\|_{L^4(\R^2)}  
\le C\|\na^2\bfu\|_{L^2(\R^2)}^{\f32}\|\na\bfu\|_{L^2(\R^2)}\|\bfu\|_{L^2(\R^2)}^{\f12},
\]
which together with \eqref{blow:nc-1}  and Young's inequality implies that
\begin{align}\label{blow:H1u-2D}
&\f12\f{d}{dt}\|\na\bfu\|_{L^2(\R^2)}^2  +  \|\na^2\bfu\|_{L^2(\R^2)}^2   \nn \\
&\leq C \kappa \|\na^2\bfu\|_{L^2(\R^2)}^{\f32}\|\na\bfu\|_{L^2(\R^2)}\|\bfu\|_{L^2(\R^2)}^{\f12}
   +  \|\Del\bfu\|_{L^2(\R^2)}\|n\|_{L^2(\R^2)}\|\na\phi\|_{L^{\infty}(\R^2)} \nn \\
&\qquad +  C\|\Del\bfu\|_{L^2(\R^2)} \|\na n\|_{L^2(\R^2)}^{\f{2}{r_3}}\|n\|_{L^2(\R^2)}^{\f{r_3-2}{r_3}}\|\na c\|_{L^{r_3}(\R^2)}  \nn \\
&\leq \f12\|\na^2\bfu\|_{L^2(\R^2)}^2  +  C \kappa^4 \|\na\bfu\|_{L^2(\R^2)}^4\|\bfu\|_{L^2(\R^2)}^2
 +  C\|n\|_{H^1(\R^2)}^2 +  C\|n\|_{L^2(\R^2)}^2\|\na c\|_{L^{r_3}(\R^2)}^{\f{2r_3}{r_3-2}}.
\end{align}
Collecting \eqref{blow:H1c-1} and \eqref{blow:H1u-2D},  we deduce from Gronwall's inequality that
\begin{align*}
&\sup_{0\leq t\leq T^*}\lt(\|\na c\|_{L^2(\R^2)}^2+\|\na\bfu\|_{L^2(\R^2)}^2\rt)
   +   \int_0^{T^*}\lt(\|\na^2 c\|_{L^2(\R^2)}^2+\|\na^2\bfu\|_{L^2(\R^2)}^2\rt) dt  \nn \\
&\leq C\exp\Big\{\kappa^4\sup_{0\leq t\leq T^*}\|\bfu\|_{L^2(\R^2)}^2
    \int_0^{T^*}\|\na\bfu\|_{L^2(\R^2)}^2dt\Big\} \Big(\|\na c_0\|_{L^2(\R^2)}^2 + \|\na\bfu_0\|_{L^2(\R^2)}^2  \nn \\
&\qquad    + \int_0^{T^*}\big(\|\na\bfu\|_{L^2(\R^2)}^2 + \|n\|_{H^1(\R^2)}^2\big)dt
    + \sup_{0\leq t\leq T^*}\|n\|_{L^2(\R^2)}^2\int_0^{T^*}\|\na c\|_{L^{r_3}(\R^3)}^{s_3}dt \Big) 
 \le C_2
\end{align*}
for some $C_2>0$ depending only on initial data and $T^*$, which together with  \eqref{blow:L2nu} yields the estimate of $(n,c,\bfu)$ in $L^{\infty}_tL_x^2\times L^{\infty}_tH_x^1\times L^{\infty}_tH_x^1$.

\vspace{2mm}

\textbf{Step 2}: Estimate of $(\na n,\na^2c,\na^2\bfu)$ in $L^{\infty}_tL_x^2\times L^{\infty}_tL_x^2\times L^{\infty}_tL_x^2$.

To estimate $\na n$, we  bound the two-dimensional \eqref{blow:H1n-1-1} by using the embedding property $H^1(\R^2)\hookrightarrow L^p(\R^2)$ for $2<p<\infty$ and Gagliardo-Nirenberg inequality as 
\begin{align}\label{blow:H1n-2}
&\|\na\big(\chi(c)n\na c\big)\|_{L^2(\R^2)}^2   \nn \\
&\leq C \Big(\|n\|_{L^4(\R^2)}^2\|\na c\|_{L^8(\R^2)}^4  +   \|\na n\|_{L^4(\R^2)}^2\|\na c\|_{L^4(\R^2)}^2
   +  \|n\|_{L^{\infty}(\R^2)}^2\|\na^2c\|_{L^2(\R^2)}^2\Big)  \nn   \\
&\le C \Big(\|\na n\|_{L^2(\R^2)}\|n\|_{L^2(\R^2)}\|\na c\|_{H^1(\R^2)}^4   +  \|\na^2n\|_{L^2(\R^2)}\|\na n\|_{L^2(\R^2)}\|\na c\|_{H^1(\R^2)}^2 \nn \\
&\qquad  +  \|\na^2n\|_{L^2(\R^2)}\|n\|_{L^2(\R^2)}\|\na^2c\|_{L^2(\R^2)}^2 \Big)  \nn \\
&\leq \f14\|\na^2n\|_{L^2(\R^2)}^2  +  C \Big(\|\na c\|_{H^1(\R^2)}^4\|\na n\|_{L^2(\R^2)}^2  +  \|\na c\|_{H^1(\R^2)}^4\|n\|_{L^2(\R^2)}^2\Big),
\end{align}
which together with the two-dimensional \eqref{blow:H1n-1} and \eqref{blow:H1n-1-2} shows that 
\begin{align*}
&\f{d}{dt}\|\na n\|_{L^2(\R^2)}^2   + \|\na^2 n\|_{L^2(\R^2)}^2  \nn \\
&\leq C\lt(1+ \|\bfu\|_{H^2(\R^2)}^2  +  \|\na c\|_{H^1(\R^2)}^4  \rt)\|\na n\|_{L^2(\R^2)}^2
   +  C\lt(1+ \|\na c\|_{H^1(\R^2)}^4 \rt)\|n\|_{L^2(\R^2)}^2
\end{align*}
and thus that 
\begin{align}\label{blow:H1n-2D}
&\sup_{0\leq t\leq T^*}\|\na n\|_{L^2(\R^2)}^2
   +   \int_0^{T^*}\|\na^2n\|_{L^2(\R^2)}^2 dt  \nn \\
&\leq C\exp\Big\{\int_0^{T^*}\lt(1+ \|\bfu\|_{H^2(\R^2)}^2  +  \|\na c\|_{H^1(\R^2)}^4  \rt)dt \Big\}     \nn \\
&\qquad   \cdot \Big(\|\na n_0\|_{L^2(\R^2)}^2+ \sup_{0\leq t\leq T^*}\|n\|_{L^2(\R^2)}^2\int_0^{T^*}\Big(1+ \|\na c\|_{H^1(\R^2)}^4 \Big)dt \Big)
\leq C_4,
\end{align}
where $C_4$ is a positive constant depending only on initial data and $T^*$.

For the estimate of $\na^2 c$, we apply H\"{o}lder's inequality and Gagliardo-Nirenberg inequality to replace  \eqref{I-1-1} by 
\begin{align*}
&\int_{\R^2}\na\Del c\cd\na(\bfu\cd\na c)  \nn \\
&\leq \f14\|\na^3c\|_{L^2(\R^2)}^2 + C\|\na\bfu\|_{L^4(\R^2)}^2\|\na c\|_{L^4(\R^2)}^2  +  C\|\bfu\|_{L^{\infty}(\R^2)}^2\|\na^2c\|_{L^2(\R^2)}^2 \nn \\
&\leq \f14\|\na^3c\|_{L^2(\R^2)}^2+ C\|\na\bfu\|_{H^1(\R^2)}^2\|\na^2c\|_{L^2(\R^2)}\|c\|_{L^{\infty}(\R^2)}
  +  C\|\bfu\|_{H^2(\R^2)}^2\|\na^2c\|_{L^2(\R^2)}^2   \nn \\
&\leq \f14\|\na^3c\|_{L^2(\R^2)}^2+ C\|\bfu\|_{H^2(\R^2)}^2\|\na^2c\|_{L^2(\R^2)}^2 + C\|c\|_{L^{\infty}(\R^2)}^2\|\na\bfu\|_{H^1(\R^2)}^2,
\end{align*} 
which together with \eqref{blow:nc-1} and Young's inequality gives that
\begin{align*}
&\f12\f{d}{dt}\|\na^2c\|_{L^2(\R^2)}^2  +  \|\na^3c\|_{L^2(\R^2)}^2   \nn \\
&\leq \f12\|\na^3c\|_{L^2(\R^2)}^2   +  C\|\bfu\|_{H^2(\R^2)}^2\|\na^2c\|_{L^2(\R^2)}^2  +  Cc_{\infty}^2\|\bfu\|_{H^2(\R^2)}^2    \nn \\
&\qquad +  C\|\na n\|_{L^2(\R^2)}^2  +  C\|\na n\|_{L^2(\R^2)}^{\f{4}{r_3}}\|n\|_{L^2(\R^2)}^{\f{2(r_3-2)}{r_3}}\|\na c\|_{L^{r_3}(\R^2)}^2 \nn\\
&\leq \f12\|\na^3c\|_{L^2(\R^2)}^2   +  C\Big(\|\bfu\|_{H^2(\R^2)}^2\|\na^2c\|_{L^2(\R^2)}^2  +  \|\bfu\|_{H^2(\R^2)}^2    +  \|\na n\|_{L^2(\R^2)}^2  \\ 
& \qquad +  \|n\|_{L^2(\R^2)}^2\|\na c\|_{L^{r_3}(\R^2)}^{\f{2r_3}{r_3-2}}\Big).
\end{align*}
It then follows from  Gronwall's inequality that
\begin{align}\label{blow:H2c}
&\sup_{0\leq t\leq T^*}\|\na^2 c\|_{L^2(\R^2)}^2
   +   \int_0^{T^*}\|\na^3 c\|_{L^2(\R^2)}^2 dt  \nn \\
&\leq C\exp\Big\{\int_0^{T^*}\|\bfu\|_{H^2(\R^2)}^2dt\Big\}   \Big(\|\na^2 c_0\|_{L^2(\R^2)}^2
+ \int_0^{T^*}\Big(\|\bfu\|_{H^2(\R^2)}^2 + \|\na n\|_{L^2(\R^2)}^2\Big)dt  \nn \\
&\qquad    + \sup_{0\leq t\leq T^*}\|n\|_{L^2(\R^2)}^2\int_0^{T^*}\|\na c\|_{L^{r_3}(\R^3)}^{s_3}dt \Big)
\leq C_3,
\end{align}
where $C_3$ is a positive constant depending only on initial data and $T^*$.

As for the estimate of $\na^2\bfu$ in the current setting, we can use Gagliardo-Nirenberg inequality, Young's inequality, \eqref{blow:H1n-2} and \eqref{blow:H1n-1-2} to deduce that
\begin{align*}
&\f12\f{d}{dt}\|\na^2\bfu\|_{L^2(\R^2)}^2   +  \|\na^3\bfu\|_{L^2(\R^2)}^2  \nn \\
&\leq \f12\|\na^3\bfu\|_{L^2(\R^2)}^2 + C\Big( \kappa^2\|\na(\bfu\cd\na\bfu)\|_{L^2(\R^2)}^2 +  \|\na(n\na\phi)\|_{L^2(\R^2)}^2
   +    \|\na\big(\chi(c)n\na c\big)\|_{L^2(\R^2)}^2\Big)  \nn \\
&\leq \f12\|\na^3\bfu\|_{L^2(\R^2)}^2 + C\Big(\kappa^2\|\na\bfu\|_{L^4(\R^2)}^4  + \kappa^2\|\bfu\|_{L^{\infty}(\R^2)}^2\|\na^2\bfu\|_{L^2(\R^2)}^2 \nn\\
&\qquad   +   \|n\|_{H^2(\R^2)}^2    +  \|\na c\|_{H^1(\R^2)}^4\|n\|_{H^1(\R^2)}^2 \Big) \nn \\
&\le \f12\|\na^3\bfu\|_{L^2(\R^2)}^2 + C \Big(\kappa^2\|\bfu\|_{H^2(\R^2)}^2\|\na^2\bfu\|_{L^2(\R^2)}^2  + \|n\|_{H^2(\R^2)}^2    +  \|\na c\|_{H^1(\R^2)}^4\|n\|_{H^1(\R^2)}^2\Big),
\end{align*}
which implies by Gronwall's inequality that 
\begin{align}\label{blow:H2u-2}
&\sup_{0\leq t\leq T^*}\|\na^2\bfu\|_{L^2(\R^2)}^2
   +   \int_0^{T^*}\|\na^3\bfu\|_{L^2(\R^2)}^2 dt  \nn \\
&\leq C\exp\Big\{\kappa^2 \int_0^{T^*}\|\bfu\|_{H^2(\R^2)}^2dt\Big\}   \Big(\|\na^2\bfu_0\|_{L^2(\R^2)}^2 + \int_0^{T^*}\|n\|_{H^2(\R^2)}^2dt\nn \\
&\qquad    +  T^*\sup_{0\leq t\leq T^*}\|\na c\|_{H^1(\R^2)}^4\sup_{0\leq t\leq T^*}\|n\|_{H^1(\R^2)}^2 \Big)
\leq C_5,
\end{align}
where $C_5$ is a positive constant depending only on initial data and $T^*$. Combining \eqref{blow:H1n-2D}, \eqref{blow:H2c} and \eqref{blow:H2u-2}, we obtain the estimate of $(\na n,\na^2c,\na^2\bfu)$ in $L^{\infty}_tL_x^2\times L^{\infty}_tL_x^2\times L^{\infty}_tL_x^2$.

\vspace{2mm}

\textbf{Step 3:} Estimate of $(n,c,\bfu)$ in $L^{\infty}_tH_x^{m-1}\times L^{\infty}_tH_x^m\times L^{\infty}_tH_x^m$ $(m\geq3)$.

Repeating  the estimates of Step 3 in the proof of Theorem \ref{blowup1} enables us to conclude that if \eqref{blow:no-3} is true, then
\[
(n, c, \bfu) \in L^{\infty}\big(0, T^*; H^{m-1}(\R^2)\times H^m(\R^2)\times H^m(\R^2; \R^2)\big)
\]
and
\[
(\na n,\na c, \na\bfu)\in L^2\big(0, T^*; H^{m-1}(\R^2)\times H^m(\R^2) \times H^m(\R^2; \R^2)\big), 
\]
which is contradictory to the assumption that $T^*$ is the maximal time of existence with $T^*<\infty$. This completes the proof of Theorem \ref{blowup2}.
\end{proof}

\vspace{2mm}

\section{Global existence of classical solution}\label{globalexist}

In this section, we shall explore the global existence of unique classical solution to the Cauchy problem of chemotaxis-Stokes system \eqref{EQ}-\eqref{initial} in $\R^3$ and chemotaxis-Navier-Stokes system \eqref{EQ}-\eqref{initial} in $\R^2$. The key is to derive an entropy functional inequality.

As is showing in the equation \eqref{EQ}$_3$, the appearance of the force term $\chi(c)n\na c$ leads to stronger nonlinearity. It is for this reason that we need the bound estimate in the following Lemma to control this bad term when we establish an entropy inequality.

\begin{lemma}\label{Lp-n-1}
Let $0<T<\infty$, $p\in[1,+\infty)$ and $d=2,3$. Suppose that the assumptions $(\bf{A})$, $(\bf{B})$ and $(\bf{C})$ hold. There exists $\delta=\delta(p)$ such that if $\|c_0\|_{L^{\infty}(\R^d)}<\delta$, then $n(t)\in L^p(\R^d)$ for all $t\in [0,T)$ and
\begin{align}\label{Lp-n-2}
\|n(t)\|_{L^p(\R^d)}\leq C=C(p,T,\|c_0\|_{L^{\infty}(\R^d)},\|n_0\|_{L^p(\R^d)}).
\end{align}
\end{lemma}

\begin{remark}\label{remark-2}
As we said, Lemma \ref{Lp-n-1} is used to control the bad term $\chi(c)n\na c$ appeared in \eqref{EQ}$_3$. However, the boundedness of $n$ in space $L^p(\R^d)$ for all $p\in[1,+\infty)$ seems relatively strong from the perspective of the proof of Lemma \ref{entropy} below.
\end{remark}

\begin{proof}[Proof of Lemma \ref{Lp-n-1}]
The proof of this Lemma is much similar as the one in \cite[Proposition 2]{Chae14CPDE}. We give a sketch of the main steps below for completeness.

Firstly, let $g$ be a positive function satisfying $g'\ge 0$, which is to be determined later. Then one can easily deduce that
\begin{align}\label{Lp-n-3}
\f{d}{dt}\int_{\R^d}n^p g(c)
&=p\int_{\R^d}n^{p-1}g(c)\pa_t n   +  \int_{\R^d}n^p g'(c)\pa_t c   \nn \\
&=p\int_{\R^d}n^{p-1}g(c)\big(-\bfu\cd\na n+\Del n -\na\cd(\chi(c)n\na c) + \na\cd(n\na\phi)\big)   \nn \\
&\qquad  + \int_{\R^d}n^p g'(c)\big(-\bfu\cd\na c + \Del c -nf(c)\big).
\end{align}
Due to the divergence-free property of $\bfu$, we have that
\begin{align}\label{Lp-n-4}
-p\int_{\R^d}n^{p-1}g(c)\bfu\cd\na n  -   \int_{\R^d}n^p g'(c)\bfu\cd\na c
= -\int_{\R^d}g(c)\bfu\cd\na n^p  -   \int_{\R^d}n^p\bfu\cd\na g(c)     =0.
\end{align}
Substituting \eqref{Lp-n-4} into \eqref{Lp-n-3} and using integration by parts, we obtain
\begin{align}\label{Lp-n-4-1}
\f{d}{dt}\int_{\R^d}n^p g(c)
&= -p(p-1)\int_{\R^d}n^{p-2}g(c)|\na n|^2  -  2p\int_{\R^d}n^{p-1}g'(c)\na c\cd\na n   \nn \\
&\qquad   +  p(p-1)\int_{\R^d}n^{p-1}g(c)\chi(c)\na n\cd\na c  +  p\int_{\R^d}n^pg'(c)\chi(c)|\na c|^2   \nn \\
&\qquad  -p(p-1)\int_{\R^d}n^{p-1}g(c)\na n\cd\na\phi  -  p\int_{\R^d}n^pg'(c)\na c\cd\na\phi   \nn \\
&\qquad -\int_{\R^d}n^p g''(c)|\na c|^2   -  \int_{\R^d}n^{p+1}g'(c)f(c),
\end{align}
which implied by the nonnegativity of $n$, $g'$ and $f$ and the application of Young's inequality that
\begin{align}\label{Lp-n-5}
&\f{d}{dt}\int_{\R^d}n^p g(c)  + \f12p(p-1)\int_{\R^d}n^{p-2}g(c)|\na n|^2  +  \f12\int_{\R^d}n^p g''(c)|\na c|^2  + \int_{\R^d}n^{p+1}g'(c)f(c) \nn \\
& \leq \f{6p}{p-1}\int_{\R^d}n^p\f{(g'(c))^2}{g(c)}|\na c|^2   +  \f{3p(p-1)}{2}\int_{\R^d}n^p\chi^2(c)g(c)|\na c|^2
    +  p\int_{\R^d}n^pg'(c)\chi(c)|\na c|^2  \nn \\
&\qquad  + \f{3p(p-1)}{2}\int_{\R^d}n^p g(c)|\na\phi|^2  +  \f{p^2}{2}\int_{\R^d}n^p\f{(g'(c))^2}{g''(c)}|\na\phi|^2.
\end{align}
Setting $g(c)=e^{(\beta c)^2}$, and we aim at looking for $g(c)$ fulfilling
\[
\f{6p}{p-1}\f{(g'(c))^2}{g(c)} + \f{3p(p-1)}{2}\chi^2(c)g(c) + pg'(c)\chi(c)  \leq \f14g''(c),
\]
which is equivalent to
\begin{align}\label{Lp-n-6}
\f{24p}{p-1}\beta^4c^2 + \f{3p(p-1)}{2}\chi^2(c) + 2p\beta^2c\,\chi(c)
\leq \f12\beta^2 + \beta^4c^2.
\end{align}
To make \eqref{Lp-n-6} be satisfied, we can firstly choose $\beta$ such that
\begin{align}\label{Lp-n-6-1}
9p(p-1)\mathcal{C}_{\chi}^2 \leq \beta^2.
\end{align}
Then, due to $\|c_0\|_{L^{\infty}(\R^d)}$ is sufficiently small, we can thus find $\delta=\delta(p)$ satisfying that for all $\|c_0\|_{L^{\infty}(\R^d)}<\delta$,
\begin{align}\label{Lp-n-6-2}
\mathcal{C}_{\chi}\|c_0\|_{L^{\infty}(\R^d)} \leq \f{1}{12p}  \qquad \textrm{and} \qquad  \|c_0\|_{L^{\infty}(\R^d)}^2 \leq \f{p-1}{144p\beta^2}.
\end{align}
One can easily check that \eqref{Lp-n-6-1} and \eqref{Lp-n-6-2} are enough to ensure the validity of \eqref{Lp-n-6}. Thus, we infer from \eqref{Lp-n-5} and \eqref{Lp-n-6} that
\begin{align}\label{Lp-n-7}
&\f{d}{dt}\int_{\R^d}n^p g(c)  + \f12p(p-1)\int_{\R^d}n^{p-2}g(c)|\na n|^2  +  \f14\int_{\R^d}n^p g''(c)|\na c|^2   \nn \\
& \leq \,\, \f{3p(p-1)}{2}\int_{\R^d}n^p g(c)|\na\phi|^2  +  \f{p^2}{2}\int_{\R^d}n^p\f{(g'(c))^2}{g''(c)}|\na\phi|^2   \nn \\
&\leq  \Big(\f{3p(p-1)}{2}   + \f{p^2}{2}\max_{0\leq c \leq c_{\infty}}\f{(g'(c))^2}{g''(c)g(c)}\Big) \|\na\phi\|_{L^{\infty}(\R^d)}^2\int_{\R^d}n^p g(c).
\end{align}
Finally, the proof of \eqref{Lp-n-2} is immediately implied by Gronwall's inequality and $e^{(\beta c)^2}>1$.
\end{proof}

With \eqref{Lp-n-2} at hand, we now present the entropy functional inequality as follows.
\begin{lemma}\label{entropy}
Let $0<T<\infty$, $d=2,3$. Suppose that the assumptions $(\bf{A})$, $(\bf{B})$ and $(\bf{C})$ hold. If $\|c_0\|_{L^{\infty}(\R^d)}$ is suitably  small, then the solution $(n,c,\bfu)$ to the Cauchy problem \eqref{EQ}-\eqref{initial} satisfies the following inequality
\begin{align*}
& \int_{\R^d} n|\ln n|   +  \|\na c\|_{L^2(\R^d)}^2  +  \|\bfu\|_{L^2(\R^d)}^2   \nn \\
&\quad +  \f12\int_0^t\Big(\int_{\R^d}\f{|\na n|^2}{n}  + \|\na^2 c\|_{L^2(\R^d)}^2   +  \|\na\bfu\|_{L^2(\R^d)}^2 \Big) ds
\leq C
\end{align*}
for all $t\in(0,T)$, where $C$ depends on $T$, $\int_{\R^d}n_0|\ln n_0|$, $\int_{\R^d}n_0\langle x\rangle $, $\|\na c_0\|_{L^2(\R^d)}^2$ and $\|\bfu_0\|_{L^2(\R^d)}^2$.
\end{lemma}
\begin{proof}
Multiplying \eqref{EQ}$_1$ by $1+\ln n$, using the integration by parts, and applying Gagliardo-Nirenberg inequality, one has
\begin{align*}
&\f{d}{dt}\int_{\R^d}n\ln n   +  \int_{\R^d}\f{|\na n|^2}{n}  \nn \\
&=  \int_{\R^d}\chi(c)\na c\cd\na n   -  \int_{\R^d}\na \phi \cd\na n\nn \\
&\leq \f12\int_{\R^d}\f{|\na n|^2}{n}   +  \mathcal{C}_{\chi}^2\|\sqrt{n}\|_{L^4(\R^d)}^2\|\na c\|_{L^4(\R^d)}^2
   +  \|\na\phi\|_{L^{\infty}(\R^d)}^2\|\sqrt{n}\|_{L^2(\R^d)}^2   \nn \\
&\leq \f12\int_{\R^d}\f{|\na n|^2}{n}   +  C\|n\|_{L^2(\R^d)}\|\na^2c\|_{L^2(\R^d)}\|c\|_{L^{\infty}(\R^d)}
   +  C\|n\|_{L^1(\R^d)} \nn \\
&\leq \f12\int_{\R^d}\f{|\na n|^2}{n}   +  \f12\|\na^2c\|_{L^2(\R^d)}^2   +  C\|n\|_{L^2(\R^d)}^2\|c\|_{L^{\infty}(\R^d)}^2  +  C\|n\|_{L^1(\R^d)},
\end{align*}
which implies that for suitably small $c_0$,
\begin{align}\label{logn-1}
\f{d}{dt}\int_{\R^d}n\ln n   +  \f12\int_{\R^d}\f{|\na n|^2}{n}
\leq  \f12\|\na^2c\|_{L^2(\R^d)}^2   +   C,
\end{align}
where we have used Lemma \ref{mass} and the boundedness of $\|n\|_{L^2(\R^d)}$ in \eqref{Lp-n-2} with $p=2$.

Similarly, taking the $L^2$-inner product of \eqref{EQ}$_2$ with $-\Del c$, we deduce that
\begin{align}\label{3D-H1c-1}
\f12\f{d}{dt}\|\na c\|_{L^2(\R^d)}^2  +  \|\na^2 c\|_{L^2(\R^d)}^2
&= \int_{\R^d}\Del c\, \bfu\cd\na c    +  \int_{\R^d}\Del c\, n \,f(c)  \nn \\
&= \int_{\R^d}(\na^2c: \na\bfu)\,c    +  \int_{\R^d}\Del c\, n \,f(c)  \nn \\
&\leq \f12\|\na^2c\|_{L^2(\R^d)}^2 + \|c\|_{L^{\infty}(\R^d)}^2\|\na \bfu\|_{L^2(\R^d)}^2  +  \mathcal{C}_{f}^2\|n\|_{L^2(\R^d)}^2.
\end{align}
Putting \eqref{logn-1} and \eqref{3D-H1c-1} together, it shows that
\begin{align}\label{logn+H1c}
& \f{d}{dt}\lt(\int_{\R^d} n\ln n  +  \|\na c\|_{L^2(\R^d)}^2\rt)  +  \f12\int_{\R^d}\f{|\na n|^2}{n}  +  \f12\|\na^2 c\|_{L^2(\R^d)}^2  \nonumber \\
& \le  2\|c_0\|_{L^{\infty}(\R^d)}^2\|\na \bfu\|_{L^2(\R^d)}^2  +  C,
\end{align}
where we have used Lemma \ref{Lp-n-1} with $p=2$ again.

In order to deal with the first term on the right-hand-side of \eqref{logn+H1c}, we test \eqref{EQ}$_3$ against $\bfu$ and apply Gagliardo-Nirenberg inequality to deduce that
\begin{align}\label{3D-L2u-1}
&\f12\f{d}{dt}\|\bfu\|_{L^2(\R^d)}^2  +  \|\na\bfu\|_{L^2(\R^d)}^2  \nn \\
&=-\int_{\R^d}\bfu\cd n\na\phi   +  \int_{\R^d}\bfu\cd\chi(c)n\na c   \nn \\
&\leq \|\bfu\|_{L^2(\R^d)}\|n\|_{L^{2}(\R^d)}\|\na\phi\|_{L^{\infty}(\R^d)}
   +  \mathcal{C}_{\chi}\|\bfu\|_{L^4(\R^d)}\|n\|_{L^4(\R^d)}\|\na c\|_{L^2(\R^d)}  \nn \\
&\leq C\|\bfu\|_{L^2(\R^d)}\|n\|_{L^{2}(\R^d)}
   +  C\|\na\bfu\|_{L^2(\R^d)}^{\f{d}{4}} \|\bfu\|_{L^2(\R^d)}^{\f{4-d}{4}} \|n\|_{L^4(\R^d)}\|\na c\|_{L^2(\R^d)}  \nn \\
&\leq \f12\|\na\bfu\|_{L^2(\R^d)}^2  +  C\|\bfu\|_{L^2(\R^d)}^2 +  C\|n\|_{L^{2}(\R^d)}^2  +  C\|n\|_{L^4(\R^d)}^2\|\na c\|_{L^2(\R^d)}^2  \nn \\
&\leq \f12\|\na\bfu\|_{L^2(\R^d)}^2  +  C\|\bfu\|_{L^2(\R^d)}^2 +  C +  C\|\na c\|_{L^2(\R^d)}^2,
\end{align}
where we have used the boundedness presented in Lemma \ref{Lp-n-1} with $p=2$ and $p=4$ respectively. Combining \eqref{logn+H1c} and \eqref{3D-L2u-1}, using the smallness assumption on $\|c_0\|_{L^{\infty}(\R^d)}$, we obtain
\begin{align}\label{logn+H1c+L2u}
& \f{d}{dt}\lt(\int_{\R^d} n\ln n  +  \|\na c\|_{L^2(\R^d)}^2  +  \|\bfu\|_{L^2(\R^d)}^2\rt)
  +  \f12\lt(\int_{\R^d}\f{|\na n|^2}{n}  +  \|\na^2 c\|_{L^2(\R^d)}^2   +  \|\na\bfu\|_{L^2(\R^d)}^2 \rt) \nn \\
& \leq C\lt(\|\na c\|_{L^2(\R^d)}^2  + \|\bfu\|_{L^2(\R^d)}^2 \rt) +  C.
\end{align}

To bound the possible negative part of $\int_{\R^d}n\ln n$, we now devote ourselves to explore the evolution estimate of the first-order spatial moment of $n$ by testing \eqref{EQ}$_1$ against $\langle x\rangle:=\sqrt{1+|x|^2}$ and using the integration by parts, and obtain
\begin{align}\label{n-moment-1}
\f{d}{dt}\int_{\R^d} n\langle x\rangle
&= \int_{\R^d}n\bfu\cd\na \langle x\rangle  +  \int_{\R^d}n\Del \langle x\rangle  +  \int_{\R^d}\chi(c)n\na c \cd \na\langle x\rangle
  -  \int_{\R^d}n\na\phi\cd\na\langle x\rangle  \nn \\
&\leq \|n\|_{L^2(\R^d)}\|\bfu\|_{L^2(\R^d)}\|\na\langle x\rangle\|_{L^{\infty}(\R^d)}  +  \|n\|_{L^1(\R^d)}\|\Del\langle x\rangle\|_{L^{\infty}(\R^d)}   \nn \\
& \qquad +  \mathcal{C}_{\chi}\|n\|_{L^2(\R^d)}\|\na c\|_{L^2(\R^d)}\|\na \langle x\rangle\|_{L^{\infty}(\R^d)}
   +  \|n\|_{L^1(\R^d)}\|\na \phi\|_{L^{\infty}(\R^d)}\|\na\langle x\rangle\|_{L^{\infty}(\R^d)}   \nn \\
&\leq C\lt(\|\na c\|_{L^2(\R^d)}^2  +  \|\bfu\|_{L^2(\R^d)}^2\rt) +  C,
\end{align}
where we have used Lemma \ref{mass}, Lemma \ref{Lp-n-1} with $p=2$ and the facts that $\|\na\langle x\rangle\|_{L^{\infty}(\R^d)}$ and $\|\Del\langle x\rangle\|_{L^{\infty}(\R^d)}$ are bounded by the definition of $\langle x\rangle$. Hence, a linear combination of \eqref{logn+H1c+L2u} and \eqref{n-moment-1} gives that
\begin{align}\label{logn+H1c+L2u+momn-1}
& \f{d}{dt}\Big(\int_{\R^d} n\ln n  +  \int_{\R^d} 2n\langle x\rangle  +  \|\na c\|_{L^2(\R^d)}^2  +  \|\bfu\|_{L^2(\R^d)}^2\Big)  \nn \\
&\qquad +  \f12\Big(\int_{\R^d}\f{|\na n|^2}{n}  +  \|\na^2 c\|_{L^2(\R^d)}^2   +  \|\na\bfu\|_{L^2(\R^d)}^2\Big)   \nn \\
&\leq C\lt(\|\na c\|_{L^2(\R^d)}^2  +  \|\bfu\|_{L^2(\R^d)}^2\rt) +  C.
\end{align}
By same reasoning for obtaining (2.27) in \cite{Duan2017JDE}, we can easily get at
\begin{align}\label{logn-2}
\int_{\R^d}\Big(n\ln n  + 2 n\langle x\rangle\Big)
\geq \int_{\R^d}n|\ln n|   -  4e^{-1}\int_{\R^d}e^{-\f12\langle x\rangle}.
\end{align}
Substituting \eqref{logn-2} into \eqref{logn+H1c+L2u+momn-1}, one can see that
\begin{align}\label{logn+H1c+L2u+momn-2}
& \f{d}{dt}\Big(\int_{\R^d} n\ln n  +  \int_{\R^d}2 n \langle x\rangle +  \|\na c\|_{L^2(\R^d)}^2  +  \|\bfu\|_{L^2(\R^d)}^2\Big)   \nn \\
&\qquad    +  \f12\Big(\int_{\R^d}\f{|\na n|^2}{n}  +  \|\na^2 c\|_{L^2(\R^d)}^2   +  \|\na\bfu\|_{L^2(\R^d)}^2 \Big)  \nn \\
&\leq C\Big(\int_{\R^d}n|\ln n|  +  \|\na c\|_{L^2(\R^d)}^2  +  \|\bfu\|_{L^2(\R^d)}^2\Big) +  C  \nn \\
&\leq C\Big(\int_{\R^d} n\ln n  +  \int_{\R^d}2 n\langle x\rangle   +  \|\na c\|_{L^2(\R^d)}^2  +  \|\bfu\|_{L^2(\R^d)}^2\Big)
  +  Ce^{-1}\int_{\R^d}e^{-\f12\langle x\rangle} + C,
\end{align}
which implied by Gronwall's inequality that
\begin{align}\label{logn+H1c+L2u+momn-3}
& \int_{\R^d} n\ln n  +  \int_{\R^d}2n\langle x\rangle   +  \|\na c\|_{L^2(\R^d)}^2  +  \|\bfu\|_{L^2(\R^d)}^2   \nn \\
&\qquad +  \f12\int_0^t\Big(\int_{\R^d}\f{|\na n|^2}{n}  +  \|\na^2 c\|_{L^2(\R^d)}^2   +  \|\na\bfu\|_{L^2(\R^d)}^2\Big) ds
\leq C(T)
\end{align}
for all $t\in(0,T)$. Using the inequality \eqref{logn-2} again, we can finally conclude from \eqref{logn+H1c+L2u+momn-3} that
\begin{align*}
& \int_{\R^d} n|\ln n|   +  \|\na c\|_{L^2(\R^d)}^2  +  \|\bfu\|_{L^2(\R^d)}^2   \nn \\
&\qquad +  \f12\int_0^t\Big(\int_{\R^d}\f{|\na n|^2}{n}  +  \|\na^2 c\|_{L^2(\R^d)}^2   +  \|\na\bfu\|_{L^2(\R^d)}^2 \Big)ds
\leq \widetilde{C}
\end{align*}
for all $t\in(0,T)$, where $\widetilde{C}$ depends on $T$, $\int_{\R^d}n_0|\ln n_0|$, $\int_{\R^d}n_0\langle x\rangle $, $\|\na c_0\|_{L^2(\R^d)}^2$ and $\|\bfu_0\|_{L^2(\R^d)}^2$. This completes the proof of Lemma \ref{entropy}.
\end{proof}

Next, with the above bound estimates obtained in Lemma \ref{Lp-n-1} and Lemma \ref{entropy} at hand, which are uniform until the maximal time of existence, we can proceed to prove that system \eqref{EQ}-\eqref{initial} in $\R^3$ with $\kappa=0$ and in $\R^2$ with $\kappa\in\R$ admits a unique regular solution globally in time with the help of the blow-up criteria showed in Theorem \ref{blowup1} and Theorem \ref{blowup2}.

\begin{proof}[Proof of Theorem \ref{Stokes}]

As is showed in \eqref{blow-1-1} of Theorem \ref{blowup1}, for the proof of global existence of regular solutions in $\R^3$ when $\kappa=0$, it is sufficient to verify that $\na c\in L^{s_1}(0,T^*;L^{r_1}(\R^3))$ for some $(r_1, s_1)$ satisfying $\f{3}{r_1}+\f{2}{s_1}\leq 1$ and $3<r_1\leq +\infty$. Here, we consider the case of $r_1=s_1=5$. Notice by using the Gagliardo-Nirenberg inequality that
\begin{align}\label{Stokes-L5c}
\|\na c\|_{L^5(\R^3)}^5 \leq C\|\na^2c\|_{L^2(\R^3)}^{4}\|c\|_{L^{\infty}(\R^3)}.
\end{align}
Thus, we only need to verify that $\na^2c\in L^{\infty}(0,T^*;L^2(\R^3))$ due to $\|c\|_{L^{\infty}(\R^3)}\leq \|c_0\|_{L^{\infty}(\R^3)}$.

For suitably small $c_0$, we use Lemma \ref{Lp-n-1} to obtain
\begin{align}\label{Stokes-Lpn}
\|n(t)\|_{L^p(\R^3)}\leq C \qquad  \textrm{for} \,\, 1\leq p<\infty
\end{align}
and Lemma \ref{entropy} to see
\begin{align}\label{Stokes-entropy}
& \int_{\R^3} n|\ln n|   +  \|\na c\|_{L^2(\R^3)}^2  +  \|\bfu\|_{L^2(\R^3)}^2   \nn \\
&\qquad +  \f12\int_0^t\lt(\int_{\R^3}\f{|\na n|^2}{n}  + \|\na^2 c\|_{L^2(\R^3)}^2   +  \|\na\bfu\|_{L^2(\R^3)}^2\rt) ds
\leq C,
\end{align}
where $C$ depends on $T^*$, $\int_{\R^3}n_0|\ln n_0|$, $\int_{\R^3}n_0\langle x\rangle $, $\|\na c_0\|_{L^2(\R^3)}^2$ and  $\|\bfu_0\|_{L^2(\R^3)}^2$.

On the other hand, multiplying $-\Del\bfu$ to both side of \eqref{EQ}$_3$ and using H\"{o}lder's inequality and Sobolev's embedding, we have
\begin{align*}
&\f12\f{d}{dt}\|\na\bfu\|_{L^2(\R^3)}^2  +  \|\na^2 \bfu\|_{L^2(\R^3)}^2   \\
&=\int_{\R^3}\Del\bfu\cd n\na\phi  -  \int_{\R^3}\Del\bfu\cd\chi(c)n\na c    \\
&\leq \f12\|\Del \bfu\|_{L^2(\R^3)}^2   +\|n\|_{L^2(\R^3)}^2\|\na\phi\|_{L^{\infty}(\R^3)}^2
   +  \mathcal{C}_{\chi}^2\|n\|_{L^3(\R^3)}^2\|\na c\|_{L^6(\R^3)}^2   \\
&\leq \f12 \|\na^2 \bfu\|_{L^2(\R^3)}^2   +  C\|n\|_{L^2(\R^3)}^2   +  C\|n\|_{L^3(\R^3)}^2\|\na^2 c\|_{L^2(\R^3)}^2,
\end{align*}
which implied by the boundedness in \eqref{Stokes-Lpn}, \eqref{Stokes-entropy}, and integration with respect to time that
\begin{align}\label{Stokes-H1u-1}
\sup_{0\leq t \leq T^*}\|\na\bfu\|_{L^2(\R^3)}^2  +  \int_0^{T^*} \|\na^2\bfu\|_{L^2(\R^3)}^2 dt
\leq C  +  C\int_0^{T^*}\|\na^2c\|_{L^2(\R^3)}^2dt  \leq C.
\end{align}

We now verify $\na^2c\in L^{\infty}(0,T^*;L^2(\R^3))$. Indeed, by applying $\na^2$ to \eqref{EQ}$_2$ and testing the resulting equation against $\na^2c$, we can deduce that
\begin{align*}
&\f12\f{d}{dt}\|\na^2c\|_{L^2(\R^3)}^2 +  \|\na^3c\|_{L^2(\R^3)}^2  \nn \\
&\leq C\|\na^3c\|_{L^2(\R^3)}\Big(\|\na\bfu\|_{L^3(\R^3)}\|\na c\|_{L^6(\R^3)}  +  \|\bfu\|_{L^{\infty}(\R^3)}\|\na^2c\|_{L^2(\R^3)}  \nn \\
&\qquad    +  \|\na n\|_{L^2(\R^3)}  +  \|n\|_{L^3(\R^3)}\|\na c\|_{L^6(\R^3)}\Big)   \nn \\
&\leq \f12\|\na^3c\|_{L^2(\R^3)}^2  +  C\Big(\|\bfu\|_{H^2(\R^3)}^2 + \|n\|_{L^3(\R^3)}^2\Big)\|\na^2c\|_{L^2(\R^3)}^2  + C\|\na n\|_{L^2(\R^3)}^2
\end{align*}
and thus that
\begin{align}\label{Stokes-H2c-1}
\f{d}{dt}\|\na^2c\|_{L^2(\R^3)}^2 +  \|\na^3c\|_{L^2(\R^3)}^2
\leq C\Big(1+\|\bfu\|_{H^2(\R^3)}^2\Big)\|\na^2c\|_{L^2(\R^3)}^2  +  C\|\na n\|_{L^2(\R^3)}^2.
\end{align}
For the term on the rightmost of \eqref{Stokes-H2c-1}, one can infer from \eqref{Lp-n-7} that $\na n\in L^2(0,T^*;  L^2(\R^3))$ by using Gronwall's inequality and specially choosing $p=2$. This together with the boundedness in \eqref{Stokes-Lpn}-\eqref{Stokes-H1u-1} and Gronwall's inequality to \eqref{Stokes-H2c-1} gives that
\begin{align}\label{Stokes-H2c-2}
&\sup_{0\leq t \leq T^*}\|\na^2 c\|_{L^2(\R^3)}^2  +  \int_0^{T^*}\|\na^3c\|_{L^2(\R^3)}^2dt  \nn \\
&\leq C\exp\Big\{\int_0^{T^*}\Big(1+\|\bfu\|_{H^2(\R^3)}^2\Big)dt\Big\}\Big(\|\na^2 c_0\|_{L^2(\R^3)}^2  + \int_0^{T^*}\|\na n\|_{L^2(\R^3)}^2dt \Big)
\leq C.
\end{align}

Collecting all the above estimates obtained in \eqref{Stokes-L5c}-\eqref{Stokes-H2c-2}, we can finally conclude that $\na c\in L^5(0,T^*;L^5(\R^3))$, which violates the assertion in \eqref{blow-1-1} if $T^*<+\infty$. This completes the proof of Theorem \ref{Stokes}.
\end{proof}

\begin{proof}[Proof of Theorem \ref{2D-NS}]
The proof of this theorem is much simpler compared to the case in $\R^3$ since the blow-up criteria in Theorem \ref{blowup2} indicates that the boundedness $\int_0^{T^*}\|\na c\|_{L^4(\R^2)}^4<\infty$ is enough to extend the local solution to a global one.

Due to the smallness assumption on $\|c_0\|_{L^{\infty}(\R^2)}$, we have the same bound estimate in Lemma \ref{Lp-n-1} and same entropy inequality in Lemma \ref{entropy}, i.e.
\[
\|n(t)\|_{L^p(\R^2)}\leq C \qquad  \textrm{for} \quad 1\leq p<\infty
\]
and
\begin{align}\label{2D-entropy}
& \int_{\R^2} n|\ln n|   +  \|\na c\|_{L^2(\R^2)}^2  +  \|\bfu\|_{L^2(\R^2)}^2   \nn \\
&\quad +  \f12\int_0^t\Big(\int_{\R^2}\f{|\na n|^2}{n}  +  \|\na^2 c\|_{L^2(\R^2)}^2   +  \|\na\bfu\|_{L^2(\R^2)}^2 \Big)ds
\leq C,
\end{align}
where $C$ depends on $T^*$, $\int_{\R^2}n_0|\ln n_0|$, $\int_{\R^2}n_0\langle x\rangle$, $\|\na c_0\|_{L^2(\R^2)}^2$, $\|\bfu_0\|_{L^2(\R^2)}^2$.
Combining the estimate \eqref{2D-entropy} and the inequality
\[
\|\na c\|_{L^4(\R^2)}\leq C\|\na^2 c\|_{L^2(\R^2)}^{\f12}\|c\|_{L^{\infty}(\R^2)}^{\f12},
\]
we have
\[
\int_0^{T^*}\|\na c\|_{L^4(\R^2)}^4dt\leq \|c_0\|_{L^{\infty}(\R^2)}^2\int_0^{T^*}\|\na^2 c\|_{L^2(\R^2)}^2dt< \infty.
\]
This completes the proof of Theorem \ref{2D-NS}.
\end{proof}

\vspace{2mm}

\section{Decay rate estimates}\label{decayestimate}

In this section, based on the global existence of unique regular solution $(n,c,\bfu)$ established in Theorem \ref{Stokes} and Theorem \ref{2D-NS}, we further prove that the obtained regular solution decays in an explicit rate.
\begin{proof}[Proof of Theorem \ref{decay}]
We will split this proof into three steps.
\vspace{3mm}\\
\textbf{Step 1}: $L^p$-decay of $c$.

In case of $2\leq p<\infty$, we multiply \eqref{EQ}$_2$ by $pc^{p-1}$ and integrate the resulting equation to obtain
\begin{align}\label{decay:Lpc-1}
\f{d}{dt}\|c\|_{L^p(\R^d)}^p  +  \f{4(p-1)}{p}\|\na c^{\f{p}{2}}\|_{L^2(\R^d)}^2 \leq 0,
\end{align}
where we have used the nonnegativity of $n,c$ and $f$. Then, multiplying \eqref{decay:Lpc-1} by $(1+t)^{\ga_1}$ with $\gamma_1>0$ to be determined and integrating over $[0,t]$, we have
\begin{align}\label{decay:Lpc-2}
&(1+t)^{\ga_1}\|c(t)\|_{L^p(\R^d)}^p  +  \f{4(p-1)}{p}\int_0^t(1+s)^{\ga_1}\|\na c^{\f{p}{2}}(s)\|_{L^2(\R^d)}^2 ds   \nn \\
&\le \|c_0\|_{L^p(\R^d)}^p  +  \ga_1\int_0^t(1+s)^{\ga_1-1}\|c(s)\|_{L^p(\R^d)}^p ds.
\end{align}
In order to deal with the last term on the right-hand-side of \eqref{decay:Lpc-2}, we use the Gagliardo-Nirenberg inequality to obtain
\begin{align}\label{decay:Lpn-8}
\|c(s)\|_{L^p(\R^d)}^p=\|c^{\f{p}{2}}(s)\|_{L^2(\R^d)}^2
\leq C\lt\|\na c^{\f{p}{2}}(s)\rt\|_{L^2(\R^d)}^{2\cd\f{d(p-1)}{d(p-1)+2}}  \lt\|c^{\f{p}{2}}(s)\rt\|_{L^{\f{2}{p}}(\R^d)}^{2\cd\f{2}{d(p-1)+2}},
\end{align}
which together with Young's inequality gives that
\begin{align}\label{decay:Lpc-3}
&\ga_1\int_0^t(1+s)^{\ga_1-1}\|c(s)\|_{L^p(\R^d)}^p ds   \nn \\
&\leq C\int_0^t(1+s)^{\ga_1-1} \lt\|\na c^{\f{p}{2}}(s)\rt\|_{L^2(\R^d)}^{2\cd\f{d(p-1)}{d(p-1)+2}}
     \lt\|c(s)\rt\|_{L^1(\R^d)}^{\f{2p}{d(p-1)+2}}ds   \nn\\
&\leq \f{2(p-1)}{p}\int_0^t(1+s)^{\ga_1}\|\na c^{\f{p}{2}}(s)\|_{L^2(\R^d)}^2 ds
   + C\int_0^t(1+s)^{\ga_1-\f{d}{2}(p-1)-1}\|c(s)\|_{L^1(\R^d)}^p ds  \nn \\
&\leq \f{2(p-1)}{p}\int_0^t(1+s)^{\ga_1}\|\na c^{\f{p}{2}}(s)\|_{L^2(\R^d)}^2 ds
   +C\|c_0\|_{L^1(\R^d)}^p(1+t)^{\ga_1-\f{d}{2}(p-1)}
\end{align}
whenever $\ga_1\geq\f{d}{2}(p-1)$, where we have used the fact that $\|c(t)\|_{L^1(\R^d)}\leq \|c_0\|_{L^1(\R^d)}$. In particular, taking $\ga_1=\f{d}{2}(p-1)$, we can obtain \eqref{decay:c} by substituting \eqref{decay:Lpc-3} into \eqref{decay:Lpc-2}.

The case $1\leq p<2$ can be proved by the interpolation inequality
\[
\|c(t)\|_{L^p(\R^d)}
\leq \lt\|c(t)\rt\|_{L^1(\R^d)}^{\f{2-p}{p}}\lt\|c(t)\rt\|_{L^2(\R^d)}^{\f{2p-2}{p}}
\leq C\big(\|c_0\|_{L^1(\R^d)\cap L^2(\R^d)}\big)(1+t)^{-\f{d}{2}(1-\f{1}{p})}.
\]
\vspace{1mm}

\noindent
\textbf{Step 2}: $L^p$-decay of $n$.

Due to the lack of a uniform bound on $\|\na c\|_{H^1(\R^d)}$ with respect to time, we will reestimate the right hand side of \eqref{Lp-n-4-1}. Indeed, by  the nonnegativity of $f$, $g$, $g'$ and $g''$, we can deduce from   \eqref{Lp-n-4-1}  that
\begin{align}\label{decay:Lpn}
&\f{d}{dt}\int_{\R^d}n^p g(c)
+p(p-1)\int_{\R^d}n^{p-2}g(c)|\na n|^2 + \int_{\R^d}n^p g''(c)|\na c|^2    \nn\\
&\leq-  2p\int_{\R^d}n^{p-1}g'(c)\na c\cd\na n     +  p(p-1)\int_{\R^d}n^{p-1}g(c)\chi(c)\na n\cd\na c \nn \\
&\qquad  +  p\int_{\R^d}n^pg'(c)\chi(c)|\na c|^2    -p(p-1)\int_{\R^d}n^{p-1}g(c)\na n\cd\na\phi  -  p\int_{\R^d}n^pg'(c)\na c\cd\na\phi.
\end{align}
For the fourth term on the right-hand-side of \eqref{decay:Lpn}, we use Young's inequality to see that
\begin{align}\label{decay:Lpn-1}
&-p(p-1)\int_{\R^d}n^{p-1}g(c)\na n\cd\na\phi   \nn \\
&\leq \f{1}{16}p(p-1)\int_{\R^d}n^{p-2}g(c)|\na n|^2  +  4p(p-1)\max_{0\leq c\leq c_{\infty}}g(c)\int_{\R^d}|n^{\f{p}{2}}\na\phi|^2.
\end{align}
Since
\begin{equation*}
\om(x)=\left\{
\begin{split}
&|x|,  \qquad x\in\R^3, \\
&(1+|x|)(1+\ln(1+|x|)), \quad x\in\R^2,
\end{split}
\right.
\end{equation*}
it follows from Hardy's inequality (see \cite{Duan2010CPDE}) that
\begin{equation}\label{Hardy}
\int_{\R^d}\f{|n|^2}{|\om(x)|^2} \leq C\int_{\R^d}|\na n|^2.
\end{equation}
Recalling $\mathcal{M}_{\om\phi}:=\sup_{x\in\R^d}\big(|\om(x)|\,|\na\phi(x)|\big)^2$ is small enough, we can bound the term on the rightmost of \eqref{decay:Lpn-1} as
\begin{align}\label{decay:Lpn-2}
&4p(p-1)\max_{0\leq c\leq c_{\infty}}g(c)\int_{\R^d}|n^{\f{p}{2}}\na\phi|^2   \nn\\
&=4p(p-1)\max_{0\leq c\leq c_{\infty}}g(c)\int_{\R^d}\Big|\f{n^{\f{p}{2}}}{\om(x)}\Big|^2\big(|\om(x)|\,|\na\phi(x)|\big)^2   \nn\\
&\leq C\mathcal{M}_{\om\phi}\max_{0\leq c\leq c_{\infty}}g(c)\int_{\R^d}|\na n^{\f{p}{2}}|^2   \nn \\
&\leq C\mathcal{M}_{\om\phi}\max_{0\leq c\leq c_{\infty}}g(c)\max_{0\leq c\leq c_{\infty}}\lt(\f{1}{g(c)}\rt)
   \int_{\R^d}n^{p-2}g(c)|\na n|^2   \nn \\
&\leq \f{1}{16} p(p-1)\int_{\R^d}n^{p-2}g(c)|\na n|^2
\end{align}
and thus
\begin{align}\label{decay:Lpn-re1}
&-p(p-1)\int_{\R^d}n^{p-1}g(c)\na n\cd\na\phi
\leq \f{1}{8}p(p-1)\int_{\R^d}n^{p-2}g(c)|\na n|^2.
\end{align}
For the fifth term on the right-hand-side of \eqref{decay:Lpn}, we can use \eqref{decay:Lpn-2} to obtain
\begin{align}\label{decay:Lpn-3}
&-  p\int_{\R^d}n^pg'(c)\na c\cd\na\phi   \nn \\
&\leq \f14\int_{\R^d} n^pg''(c)|\na c|^2
  +  p^2\int_{\R^d} \f{(g'(c))^2}{g''(c)}|n^{\f{p}{2}}\na\phi|^2  \nn \\
&\leq \f14\int_{\R^d} n^pg''(c)|\na c|^2
  + C\mathcal{M}_{\om\phi}\max_{0\leq c\leq c_{\infty}}\lt(\f{(g'(c))^2}{g''(c)}\rt)\max_{0\leq c\leq c_{\infty}}\lt(\f{1}{g(c)}\rt)\int_{\R^d}n^{p-2}g(c)|\na n|^2  \nn \\
&\leq \f14\int_{\R^d} n^pg''(c)|\na c|^2 + \f{1}{8} p(p-1)\int_{\R^d}n^{p-2}g(c)|\na n|^2
\end{align}
whenever $\mathcal{M}_{\om\phi}$ is suitably small. Substituting \eqref{decay:Lpn-re1} and \eqref{decay:Lpn-3} into \eqref{decay:Lpn} and using Young's inequality, we have
\begin{align}\label{decay:Lpn-4}
&\f{d}{dt}\int_{\R^d}n^p g(c)  + \f12p(p-1)\int_{\R^d}n^{p-2}g(c)|\na n|^2  +  \f34\int_{\R^d}n^p g''(c)|\na c|^2   \nn \\
& \leq \f{8p}{p-1}\int_{\R^d}n^p\f{(g'(c))^2}{g(c)}|\na c|^2   +  2p(p-1)\int_{\R^d}n^p\chi^2(c)g(c)|\na c|^2
    +  p\int_{\R^d}n^pg'(c)\chi(c)|\na c|^2.
\end{align}
Similar as in \eqref{Lp-n-6}-\eqref{Lp-n-6-2}, we take  $g(c)=e^{(\beta c)^2}$ and aim to have
\begin{align*}
\f{8p}{p-1}\f{(g'(c))^2}{g(c)} + 2p(p-1)\chi^2(c)g(c)  +  pg'(c)\chi(c)  \leq \f14g''(c),
\end{align*}
which is equivalent to
\begin{align}\label{decay:Lpn-g}
\f{32p}{p-1}\beta^4c^2 + 2p(p-1)\chi^2(c) + 2p\beta^2c\,\chi(c)
\leq \f12\beta^2 + \beta^4c^2.
\end{align}
Indeed, by choosing $\beta$ to be suitably large and using the smallness assumption on $\|c_0\|_{L^{\infty}(\R^d)}$, we have
\begin{align*}
12p(p-1)\mathcal{C}_{\chi}^2  \le \beta^2, \qquad
\mathcal{C}_{\chi}\|c_0\|_{L^{\infty}(\R^d)} \le \f{1}{12p} \qquad \textrm{and} \qquad
  \|c_0\|_{L^{\infty}(\R^d)}^2\le \f{p-1}{192p\beta^2},
\end{align*}
which entail the validation of \eqref{decay:Lpn-g} and thus
\begin{align}\label{decay:Lpn-5}
\f{d}{dt}\int_{\R^d}n^p g(c)  +  \f12p(p-1)\int_{\R^d}n^{p-2}g(c)|\na n|^2  \leq 0.
\end{align}
Then multiplying \eqref{decay:Lpn-5} by $(1+t)^{\ga_2}$ with $\gamma_2>0$ to be determined and integrating over $[0,t]$, one has
\begin{align*}
&(1+t)^{\ga_2}\int_{\R^d}n^p g(c)  +  \f12p(p-1)\int_0^t(1+s)^{\ga_2}\int_{\R^d}n^{p-2}g(c)|\na n|^2 ds  \nn \\
&\le \int_{\R^d}n_0^p\, g(c_0)  +  \ga_2\int_0^t(1+s)^{\ga_2-1}\int_{\R^d}n^p g(c) \,ds.
\end{align*}
Noticing that
\begin{align*}
\|\na n^{\f{p}{2}}\|_{L^2(\R^d)}^2 = \f{p^2}{4}\int_{\R^d}n^{p-2}|\na n|^2
\leq \f{p^2}{4}\max_{0\leq c\leq c_{\infty}}\lt(\f{1}{g(c)}\rt)\int_{\R^d}n^{p-2}g(c)|\na n|^2
\end{align*}
and thus that
\begin{align*}
\int_{\R^d}n^p g(c)&\leq \max_{0\leq c\leq c_{\infty}}g(c)\|n\|_{L^p(\R^d)}^p   \nn \\
&\leq C\max_{0\leq c\leq c_{\infty}}g(c)\lt\|\na n^{\f{p}{2}}\rt\|_{L^2(\R^d)}^{2\cd\f{d(p-1)}{d(p-1)+2}}  \lt\|n^{\f{p}{2}}\rt\|_{L^{\f{2}{p}}(\R^d)}^{2\cd\f{2}{d(p-1)+2}}  \nn \\
&\leq C\max_{0\leq c\leq c_{\infty}}g(c)\lt(\max_{0\leq c\leq c_{\infty}}\lt(\f{1}{g(c)}\rt)\int_{\R^d}n^{p-2}g(c)|\na n|^2\rt)^{\f{d(p-1)}{d(p-1)+2}}  \lt\|n\rt\|_{L^1(\R^d)}^{\f{2p}{d(p-1)+2}},
\end{align*}
we can apply H\"{o}lder's inequality and Young's inequality to obtain
\begin{align}\label{decay:Lpn-9}
&(1+t)^{\ga_2}\int_{\R^d}n^p g(c)  +  \f14p(p-1)\int_0^t(1+s)^{\ga_2}\int_{\R^d}n^{p-2}g(c)|\na n|^2 ds  \nn \\
&\leq \int_{\R^d}n_0^p \,g(c_0)  +  C\int_0^t(1+s)^{\ga_2-\f{d}{2}(p-1)-1}\|n\|_{L^1(\R^d)}^p ds   \nn \\
&\leq \int_{\R^d}n_0^p\, g(c_0)  +  C\|n_0\|_{L^1(\R^d)}^p (1+s)^{\ga_2-\f{d}{2}(p-1)}.
\end{align}
Taking $\ga_2=\f{d}{2}(p-1)$ and using the facts that $g(c)>1$ and $g(c_0)$ is bounded, one can infer from \eqref{decay:Lpn-9} that \eqref{decay:n} is hold for $2\le p<\infty$, while  \eqref{decay:n} for $1\leq p<2$ follows from a direct interpolation between $p=1$ and $p=2$.

\vspace{2mm}

\noindent
\textbf{Step 3}: $L^{\infty}$-decay of $c$.

For simplicity, we set
\[
\xi(t):=(1+t)^{-1}, \qquad \eta(t):=(1+t)^{-\f{d}{4}}.
\]
Letting $(c-\tau\eta(t))_{+}:=\max\{c-\tau\eta(t),0\}$, differentiating $\f12\int_{\R^d}(c-\tau\eta(t))_{+}^2$ with respect to time variable $t$, then using the nonnegativity of $n$ and $f$ and the divergence-free property of $\bfu$, we have
\begin{align*}
\f12\f{d}{dt}\int_{\R^d}(c-\tau\eta(t))_{+}^2
&=\int_{\R^d}(c-\tau\eta(t))_{+}\pa_t c  -  \tau\eta'(t)\int_{\R^d}(c-\tau\eta(t))_{+}  \nn \\
&=\int_{\R^d}(c-\tau\eta(t))_{+}\big(-\bfu\cd\na c+\Del c -nf(c)\big)  -  \tau\eta'(t)\int_{\R^d}(c-\tau\eta(t))_{+}  \nn \\
&\leq-\int_{\R^d}|\na(c-\tau\eta(t))_{+}|^2 -  \tau\eta'(t)\int_{\R^d}(c-\tau\eta(t))_{+}
\end{align*}
for any $\tau>0$, which implied by a direct integration with respect to time that
\begin{align}\label{decay:c-inf-1}
&\sup_{0\leq t\leq T}\f12\int_{\R^d}(c-\tau\eta(t))_{+}^2 +  \int_0^T\int_{\R^d}|\na(c-\tau\eta(t))_{+}|^2  \nn \\
&\,\,\leq \f12\int_{\R^d}(c_0-\tau)_{+}^2  -  \tau \int_0^T \eta'(t)\int_{\R^d}(c-\tau\eta(t))_{+}
\end{align}
for any fixed $T>0$. For notational simplicity, we will also denote
\begin{align*}
\mathcal{E}(\tau):=\sup_{0\leq t\leq T}\f12\int_{\R^d}(c-\tau\eta(t))_{+}^2 +  \int_0^T\int_{\R^d}|\na(c-\tau\eta(t))_{+}|^2
\end{align*}
and
\begin{align*}
\mathcal{U}(\tau):=\int_0^T\xi(t)\int_{\R^d}(c-\tau\eta(t))_{+}^2,\qquad
\mathcal{A}(\tau):=\int_0^T\xi(t)^{\f{d+4}{4}}\int_{\R^d}(c-\tau\eta(t))_{+}.
\end{align*}
For any $\tau\ge \tau_0:=\|c_0\|_{L^{\infty}(\R^d)}$, a direct calculation  shows that
\begin{align}\label{decay:c-inf-4}
\mathcal{U}'(\tau)
= -2\int_0^T\xi(t)\eta(t)\int_{\R^d}(c-\tau\eta(t))_{+}
= - 2\mathcal{A}(\tau)
\end{align}
due to $\xi(t)\eta(t)=\xi(t)^{\f{d+4}{4}}$,  which together with \eqref{decay:c-inf-1} and $|\eta'(t)|=\f{d}{4}\xi(t)\eta(t)$ implies that
\begin{align}\label{decay:c-inf-2}
\mathcal{E}(\tau)
\le \tau \int_0^T \big| \eta'(t)\big| \int_{\R^d}(c-\tau\eta(t))_{+}
=  \f{d\tau}{4} \int_0^T \xi(t)\eta(t)\int_{\R^d}(c-\tau\eta(t))_{+}
\le  \f{d\tau}{8}\big|\mathcal{U}'(\tau)\big|.
\end{align}
On the other hand, an application of the interpolation and Gagliardo-Nirenberg inequality  yields that
\begin{align*}
\|(c-\tau\eta(t))_{+}\|_{L^2(\R^d)}^2
&\le \|(c-\tau\eta(t))_{+}\|_{L^1(\R^d)}^{\f{4}{d+4}}\|(c-\tau\eta(t))_{+}\|_{L^{\f{2(d+2)}{d}}(\R^d)}^{\f{2(d+2)}{d+4}}   \nn \\
&\le  C \|(c-\tau\eta(t))_{+}\|_{L^1(\R^d)}^{\f{4}{d+4}} \|\na(c-\tau\eta(t))_{+}\|_{L^2(\R^d)}^{\f{2d}{d+4}}
   \|(c-\tau\eta(t))_{+}\|_{L^2(\R^d)}^{\f{4}{d+4}}
\end{align*}
for some constant $C$, which together with H\"{o}lder's inequality entails that
\begin{align}\label{decay:nc-9}
\mathcal{U}(\tau)
& \le C \int_0^T\xi(t) \|(c-\tau\eta(t))_{+}\|_{L^1(\R^d)}^{\f{4}{d+4}} \|\na(c-\tau\eta(t))_{+}\|_{L^2(\R^d)}^{\f{2d}{d+4}}  \|(c-\tau\eta(t))_{+}\|_{L^2(\R^d)}^{\f{4}{d+4}}  \nonumber \\
& \le C \Big(\int_0^T\xi(t)^{\f{d+4}{4}} \|(c-\tau\eta(t))_{+}\|_{L^1(\R^d)}\Big)^{\f{4}{d+4}}  \|\na(c-\tau\eta(t))_{+}\|_{L^2L^2}^{\f{2d}{d+4}}  \|(c-\tau\eta(t))_{+}\|_{L^{\infty}L^2}^{\f{4}{d+4}}  \nonumber \\
& \le C\mathcal{A}(\tau)^{\f{4}{d+4}}\mathcal{E}(\tau)^{\f{d+2}{d+4}}.
\end{align}
Now we can deduce from  \eqref{decay:c-inf-4}, \eqref{decay:nc-9} and \eqref{decay:c-inf-2} that
\[
|\mathcal{U}'(\tau)| = 2\mathcal{A}(\tau)
\geq C\mathcal{U}(\tau)^{\f{d+4}{4}}\mathcal{E}(\tau)^{-\f{d+2}{4}}
\geq C\tau^{-\f{d+2}{4}} \mathcal{U}(\tau)^{\f{d+4}{4}}|\mathcal{U}'(\tau)|^{-\f{d+2}{4}}
\]
and thus from the nonpositivity  of  $\mathcal{U}'(\tau)$ that
\begin{align}\label{decay:nc-11}
\mathcal{U}'(\tau)\leq -C\tau^{-\f{d+2}{d+6}}\,\mathcal{U}(\tau)^{\f{d+4}{d+6}}.
\end{align}
Noticing that
\[
c-\f{\tau_0}{2}\eta(t)\geq \f{\tau_0}{2}\eta(t) \qquad \textrm{on}\quad \mathcal{S}_t :=\Big\{x\in \R^d\,\big|\,c(x, t) - \tau_0\eta(t)\geq 0\Big\},
\]
we have
\begin{align*}
\mathcal{U}(\tau_0)
& = \int_0^T\xi(t)\int_{\mathcal{S}_t}(c-\tau_0\eta(t))_{+}^2   \nn \\
&\leq \int_0^T\xi(t)\int_{\mathcal{S}_t}\lt(c-\f{\tau_0}{2}\eta(t)\rt)_{+}^{\f{2(d+2)}{d}}\lt(\f{\tau_0}{2}\eta(t)\rt)^{-\f{4}{d}}  \nn \\
&\leq \lt(\f{\tau_0}{2}\rt)^{-\f{4}{d}} \int_0^T\f{\xi(t)}{\eta(t)^{\f{4}{d}}} \int_{\R^d}c^{\f{2(d+2)}{d}} \nn \\
& \le  C\tau_0^{-\f{4}{d}} \int_0^T\|\na c\|_{L^2(\R^d)}^2\|c\|_{L^2(\R^d)}^{\f{4}{d}} \nn \\
& \le C\tau_0^{-\f{4}{d}} \|c_0\|_{L^2(\R^d)}^{\f{2(d+2)}{d}},
\end{align*}
where we used  a direct result from \eqref{blow:cL2} in the last inequality. Thus we can conclude from the ODI \eqref{decay:nc-11} that
\[
\mathcal{U}(\tau)^{\f2{d+6}}
\le \mathcal{U}(\tau_0)^{\f2{d+6}} - C\Big( \tau^{\f4{d+6}} -  \tau_0^{\f4{d+6}} \Big),
\]
which implies that  $\mathcal{U}(\tau)$ must vanish at some finite point $\tau=\tau(\|c_0\|_{L^{\infty}(\R^d)},\|c_0\|_{L^2(\R^d)})$ and thus that
\[
\|c(t)\|_{L^{\infty}(\R^d)} \le C(1+t)^{-\f{d}{4}}  
\]
for any $t\in (0, T)$.  By the arbitrariness of $T$, we obtained the desired  $L^{\infty}$-decay \eqref{decay:nc}.  This completes the proof of Theorem \ref{decay}.
\end{proof}

\vspace{2mm}

\section{Appendix}\label{appendix}

In this appendix, we establish the global-in-time existence of weak solutions to the chemotaxis-Navier-Stokes system \eqref{EQ}-\eqref{initial} in $\R^3$.  The same conclusion can  be easily deduced for $\R^2$.

We begin with presenting the definition of  global weak solutions to system \eqref{EQ}-\eqref{initial}.
\begin{definition}[Weak solution]\label{def-weak}
A triple $(n,c,\bfu)$ is called a global weak solution to the Cauchy problem \eqref{EQ}-\eqref{initial} if for any given $T>0$, the following conditions are satisfied:

\begin{itemize}

  \item[(i)] it holds that $n(x, t)\geq0$, $c(x, t)\geq0$ a.e. in $\R^3\times[0,T]$, and that
  \begin{align*}
  &n(1 +  |\ln n|)\in L^{\infty}\big(0,T; L^1(\R^3)\big), \quad \na\sqrt{n}\in L^2\big(0,T; L^2(\R^3)\big),  \nn \\
  &c \in L^{\infty}\big(0,T; L^1(\R^3)\cap L^{\infty}(\R^3)\cap H^1(\R^3)\big),  \quad \Del c\in L^2\big(0,T; L^2(\R^3)\big),  \nn \\
  &\bfu\in L^{\infty}\big(0,T; L^2(\R^3;\R^3)\big) \cap L^2\big(0,T;H^1(\R^3;\R^3)\big);
  \end{align*}

  \item[(ii)]  for any $\varphi, \,\psi\in C_0^{\infty}\big(\R^3\times[0, T]\big)$ with $\varphi(\cd,T)=0$ and $\psi(\cd,T)=0$, it holds that
      \[
      \int_0^T\int_{\R^3}n\Big(\pa_t\varphi + \Del\varphi + \bfu\cd\na\varphi + \chi(c)\na c\cd\na\varphi - \na\phi \cd \na\varphi\Big) dxdt
        +  \int_{\R^3}n_0(x)\varphi(x,0) dx =0
      \]
      and
      \[
      \int_0^T\int_{\R^3} \Big(c\big(\pa_t\psi + \Del\psi + \bfu\cd\na\psi\big)  -  nf(c)\psi\Big) dxdt  +  \int_{\R^3}c_0(x)\psi(x,0) dx =0,
      \]
      and for any $\Psi\in C_0^{\infty}\big((\R^3; \R^3)\times[0, T]\big)$ with $\na\cd\Psi=0$ and $\Psi(\cd,T)=0$, it holds that
      \begin{align*}
     &  \int_0^T \int_{\R^3} \Big(\bfu\cd\big(\pa_t\Psi + \Del\Psi\big) + \kappa(\bfu\otimes\bfu):\na\Psi \Big) dxdt\nn \\
     & \qquad +\int_0^T\int_{\R^3}\big(- n\na\phi +  \chi(c)n\na c\big)\cd\Psi dxdt + \int_{\R^3}\bfu_0(x) \cd \Psi(x,0) dx = 0.
      \end{align*}
\end{itemize}
\end{definition}

The global existence of weak solutions to system \eqref{EQ}-\eqref{initial} in $\R^3$ can be stated as follows.

\begin{theorem}[Global existence of weak solutions in $\R^3$ with $\kappa\in\R$]\label{global-weak}
Let $d=3, \kappa\in\R$. Suppose that the assumptions $(\bf{A})$, $(\bf{B})$ and $(\bf{C})$ hold. If additionally $\|c_0\|_{L^{\infty}(\R^3)}$ is suitably small,  then the Cauchy problem \eqref{EQ}-\eqref{initial} admits at least a global-in-time weak solution $(n,c,\bfu)$ in the sense of Definition \ref{def-weak}.
\end{theorem}

We will give a sketch for the proof of Theorem \ref{global-weak} in Section \ref{passlimit}.

\subsection{Preliminaries}

In this subsection, we present some preliminaries. Let  $H^{-1}(\R^3)$ stand for the dual space of $H^1(\R^3)$. For notational simplicity, we also set
\[
\mathcal{V}(\R^3) :=\Big\{ \bfu \, \big| \, \bfu = (u_1, u_2, u_3) \in (H^1(\R^3))^3\Big\}, \quad
\mathcal{V}'(\R^3) := \Big\{\bfv \, \big| \, \bfv = (v_1, v_2, v_3) \in (H^{-1}(\R^3))^3\Big\},
\]
and
\[
 \mathcal{V}_{\sigma}(\R^3) :=\big\{\bfu\in\mathcal{V}(\R^3)\, \big| \, \na\cd\bfu=0 \big\}, \qquad
  \mathcal{H}(\R^3) := \text{the closure of $\mathcal{V}_{\sigma}(\R^3)\,\,\, \textrm{in}\,\,\, L^2(\R^3;\R^3)$}.
 \]
Then for $\bfu \in\mathcal{V}(\R^3)$ and $\bfv\in\mathcal{V}'(\R^3)$, the duality is defined by $\langle\bfu, \bfv\rangle :=\sum\limits_{i=1}^{3}\langle u_i, v_i\rangle_{H^1\times H^{-1}}$.  Based on these, we denote the space $\mathcal{V}_{\sigma}^{\circ}(\R^3) := \big\{\bfv\in\mathcal{V}'(\R^3) \, \big| \, \langle\bfu, \bfv\rangle=0\,\,\,  \textrm{for all} \,\,\,\bfu \in \mathcal{V}_{\sigma}(\R^3)\big\}$.

Next, let us define two operators which will be used in constructing approximate solutions in section \ref{sec:reguler}, and state their properties.

\begin{definition}[\cite{Chemin2006}]\label{operator}
  The bilinear map $Q$ is defined by
  \begin{align*}
  Q:\,\,&\mathcal{V}(\R^3)\times\mathcal{V}(\R^3) \, \rightarrow \,\mathcal{V}'(\R^3),  \nn \\
  & \qquad (\bfu, \, \bfv) \quad \,\,\,\,\, \mapsto \,\, Q (\bfu, \, \bfv) := -\na\cd(\bfu\otimes\bfv).
  \end{align*}
\end{definition}

\begin{proposition}[\cite{Chemin2006}]\label{operator-property}
There exists a family of orthogonal projections on $\mathcal{H}(\R^3)$, denoted by $(P_l)_{l\in\mathbb{N}}$, which satisfies the following property:
\begin{itemize}
\item  For any $\bfu\in\mathcal{H}(\R^3)$, the vector field $P_l \bfu$ is an element of $\mathcal{V}_{\sigma}(\R^3)$ satisfying
  \[
  \lim_{l\to + \infty}\|P_l\bfu-\bfu\|_{\mathcal{H}(\R^3)}=0
  \]
 and
  \begin{equation}\label{property-3}
  \|\na P_l\bfu\|_{L^2(\R^3)} \leq \sqrt{l}\|\bfu\|_{L^2(\R^3)}, \quad \|\Del P_l\bfu\|_{L^2(\R^3)} \leq l\|\bfu\|_{L^2(\R^3)}.
  \end{equation}
 \item For any $\bfu,\bfv\in\mathcal{V}(\R^3)$, there exists a positive constant $C$ such that for any $\varphi\in\mathcal{V}(\R^3)$
 \[
  \langle Q(\bfu,\bfv),\varphi\rangle
  \le C\|\na\bfu\|_{L^2(\R^3)}^{\f34}\|\na\bfv\|_{L^2(\R^3)}^{\f34}\|\bfu\|_{L^2(\R^3)}^{\f14}\|\bfv\|_{L^2(\R^3)}^{\f14}\|\na\varphi\|_{L^2(\R^3)}.
 \]
  Moreover, for any $\bfu\in\mathcal{V}_{\sigma}(\R^3)$ and $\bfv\in\mathcal{V}(\R^3)$,
 \[
  \langle Q(\bfu,\bfv),\bfv\rangle =0.
 \]
\end{itemize}

\end{proposition}

\subsection{Regularisation}\label{sec:reguler}

In this subsection, we investigate a regularized system to \eqref{EQ}-\eqref{initial},  which is consistent with the entropy stated in Lemma \ref{entropy}. We will construct the corresponding  approximate solution by using $(P_l)_{l\in\mathbb{Z}}$ as in  \cite[Chapter 2]{Chemin2006} for the pure Navier-Stokes system. Precisely, we regularize system \eqref{EQ} as follows:
\begin{equation}\label{regular-EQ}
\left\{
\begin{split}
 \pa_t n^{l,\ep}  +  \bfu^{l,\ep}\cd\na n^{l,\ep}  & =  \Del n^{l,\ep}  -  \na\cd\big(n^{l,\ep} \big(\chi(c^{l,\ep})\na c^{l,\ep}\big) \ast\eta^{\ep}\big)  +  \na\cd\big( (n^{l,\ep}\na\phi) \ast\eta^{\ep}\big),   \\
\pa_t c^{l,\ep}  +  \bfu^{l,\ep}\cd\na c^{l,\ep}  & =  \Del c^{l,\ep}  -  \big(n^{l,\ep}f(c^{l,\ep})\big)\ast\eta^{\ep},  \\
\pa_t\bfu^{l,\ep} + \kappa P_lQ(\bfu^{l,\ep},\bfu^{l,\ep})  & =  P_l\Del\bfu^{l,\ep} - P_l(n^{l,\ep}\na\phi)
  +  P_l\big(n^{l,\ep}\chi(c^{l,\ep})\na c^{l,\ep}\big),  \\
\na\cd\bfu^{l,\ep} & = 0
\end{split}
\right.
\end{equation}
in $\R^3\times (0, +\infty)$ with initial data
\begin{align}\label{regular-initial}
\big(n^{l,\ep}, c^{l,\ep}, \bfu^{l,\ep}\big) \big|_{t=0} = \big(n_0^{l,\ep}, c_0^{l,\ep}, \bfu_0^{l,\ep}\big)
: = \big(n_0\ast\eta^{\ep}, c_0\ast\eta^{\ep}, P_l\bfu_0\ast\eta^{\ep}\big)   \qquad \text{in}\quad \R^3,
\end{align}
where $\eta^{\ep}$ is a mollifier. For simplicity, we will denote by $\mathcal{H}_l(\R^3)$ the space $P_l\mathcal{H}(\R^3)$. Then for each given $l$ and $\ep$, one can easily prove that system \eqref{regular-EQ}-\eqref{regular-initial} admits a local solution $(n^{l,\ep},c^{l,\ep},\bfu^{l,\ep})$ in the class
\begin{align*}
& (n^{l,\ep},c^{l,\ep},\bfu^{l,\ep})\in L^{\infty}\big(0,T; \, H^2(\R^3)\times H^2(\R^3) \times \big(H^2(\R^3:\R^3)\cap \mathcal{H}_l(\R^3)\big)\big)\\
& ( n^{l,\ep}, c^{l,\ep},\bfu^{l,\ep})\in L^2\big(0,T; \, H^3(\R^3)\times H^3(\R^3)\times H^3(\R^3;\R^3)\big)
\end{align*}
for any  $T\in (0, T^*)$ with some $T^*\in (0, +\infty]$ fulfilling the property that
\begin{itemize}
  \item $n^{l,\ep}\ge 0$ \, and \, $c^{l,\ep}\ge 0$ \, in $\R^3\times(0, T^*)$;
  \item $\|n^{l,\ep}\|_{L^1(\R^3)}=\|n_0^{l,\ep}\|_{L^1(\R^3)}\le C\|n_0\|_{L^1(\R^3)}$ \,  in $(0, T^*)$;
  \item $\|c^{l,\ep}\|_{L^p(\R^3)}\leq \|c_0^{l,\ep}\|_{L^p(\R^3)}\le C \|c_0\|_{L^p(\R^3)}$  \,  in $(0, T^*)$ for any $1\leq p\leq \infty$
\end{itemize}
with some universal positive constant $C$. Moreover, we have the following energy inequality:

\begin{proposition}\label{regular-ennergy}
For any $T\in (0, T^*)$,  the solution $(n^{l,\ep},c^{l,\ep},\bfu^{l,\ep})$ of system \eqref{regular-EQ}-\eqref{regular-initial} satisfies the uniform estimate
\begin{align}\label{regular-entropy}
& \int_{\R^3} n^{l,\ep} \big|\ln n^{l,\ep}\big|   +  \|\na c^{l,\ep}\|_{L^2(\R^3)}^2  +  \|\bfu^{l,\ep}\|_{L^2(\R^3)}^2   \nn \\
&\qquad +  \f12\int_0^t\Big(\int_{\R^3}\f{|\na n^{l,\ep}|^2}{n^{l,\ep}}  +  \|\na^2 c^{l,\ep}\|_{L^2(\R^3)}^2
   +  \|\na\bfu^{l,\ep}\|_{L^2(\R^3)}^2 \Big)ds
\leq C
\end{align}
for all $t\in(0, T)$, where $C$ is a positive constant depending only  on $T$, $\int_{\R^3}n_0|\ln n_0|$, $\int_{\R^3} n_0\langle x\rangle$, $\|\na c_0\|_{L^2(\R^3)}^2$ and $\|\bfu_0\|_{L^2(\R^3)}^2$.
\end{proposition}
\begin{proof} The proof is similar to that of Lemma \ref{entropy}. We only remark that the definitions of $n_0^{l,\ep}$ and $c_0^{l,\ep}$, and the convexity of $x\ln x$ imply that
\[
\|\na c_0^{l,\ep}\|_{L^2(\R^3)}\leq \|\na c_0\|_{L^2(\R^3)}, \qquad
\int_{\R^3}n_0^{l,\ep}(\ln n_0^{l,\ep})_{+}  \leq \int_{\R^3}n_0|\ln n_0|,
\]
and
\[
\int_{\R^3}n_0^{l,\ep}\langle x\rangle   \le \int_{\R^3}n_0\langle x\rangle ,\qquad
\int_{\R^3}n_0^{l,\ep}(\ln n_0^{l,\ep})_{-}  \le C\Big(1+\int_{\R^3}n_0\langle x\rangle\Big),
\]
which entail the dependence on the initial data of constant $C$.
\end{proof}

With the energy estimate \eqref{regular-entropy} at hand, we can now extend the obtained local solution of \eqref{regular-EQ}-\eqref{regular-initial} to a global one.
\begin{proposition}\label{regular-global}
The regularized system \eqref{regular-EQ}-\eqref{regular-initial} admits a unique global-in-time classical solution $(n^{l,\ep},c^{l,\ep},\bfu^{l,\ep})$.
\end{proposition}
\begin{proof}
Noticing that the same regularity criterion \eqref{blow-1} in Theorem \ref{blowup1} can be established  for system \eqref{regular-EQ}-\eqref{regular-initial} by repeating  the proof of Theorem \ref{blowup1}, we just need to apply this  regularity criterion to the local solution $(n^{l,\ep},c^{l,\ep},\bfu^{l,\ep})$.

On the one hannd,  since $\|\bfu^{l,\ep}\|_{L^{\infty}(0,T; L^2(\R^3))}$ is uniformly bounded in $T\in (0, T^*)$ by  \eqref{regular-entropy}, we can infer from the inequality
\[
\|\na\bfu^{l,\ep}\|_{L^2(\R^3)}
= \|\na P_l\bfu^{l,\ep}\|_{L^2(\R^3)}
\le \sqrt{l}\|\bfu^{l,\ep}\|_{L^2(\R^3)}
\]
due to \eqref{property-3}  that
\begin{align*}
\|\bfu^{l,\ep}\|_{L^5\big(0,T; L^5(\R^3)\big)}
& \le C\|\bfu^{l,\ep}\|_{L^5\big(0,T; H^1(\R^3)\big)} \\
& \le CT^{\f15} \|\bfu^{l,\ep}\|_{L^{\infty}\big(0, T^*; H^1(\R^3)\big)} \\
& \le CT^{\f15}\big(\sqrt{l}+1\big) \|\bfu^{l,\ep}\|_{L^{\infty}\big(0,T; L^2(\R^3)\big)},
 \end{align*}
which is uniformly bounded in $T\in (0, T^*)$ if the maximal time of existence $T^*<+\infty$,  which will contradict to the Serrin condition in \eqref{blow-1} of $\bfu^{l,\ep}$ by choosing $r_2=s_2=5$.

  On the other hand, for  $c^{l,\ep}$, we can  further deduce from the uniform boundedness of  $\|\na c^{l,\ep}\|_{L^2(0,T; H^1(\R^3))}$ due to \eqref{regular-entropy} that
\begin{align*}
&\f12\f{d}{dt}\|\na^2c^{l,\ep}\|_{L^2(\R^3)}^2  +  \|\na^3c^{l,\ep}\|_{L^2(\R^3)}^2   \\
&\le \f14 \|\na^3c^{l,\ep}\|_{L^2(\R^3)}^2   +  C\Big( \|\na(\bfu^{l,\ep}\cd\na c^{l,\ep})\|_{L^2(\R^3)}^2
   +  \|\na\big((n^{l,\ep}f(c^{l,\ep}))\ast\eta^{\ep}\big)\|_{L^2(\R^3)}^2 \Big)  \\
&\le \f14 \|\na^3c^{l,\ep}\|_{L^2(\R^3)}^2   +  C\Big( \|\na \bfu^{l,\ep}\|_{L^2(\R^3)}^2 \|\na c^{l,\ep}\|_{L^{\infty}(\R^3)}^2 + \|\bfu^{l,\ep}\|_{L^6(\R^3)}^2\|\na^2 c^{l,\ep}\|_{L^3(\R^3)}^2
   +  \|n^{l,\ep}\|_{L^1(\R^3)}^2\Big)   \\
 &\le \f14 \|\na^3c^{l,\ep}\|_{L^2(\R^3)}^2   +  C\Big( \|\na \bfu^{l,\ep}\|_{L^2(\R^3)}^2 \|\na^3 c^{l,\ep}\|_{L^2(\R^3)}  \|\na^2 c^{l,\ep}\|_{L^2(\R^3)}
   +  \|n_0^{l,\ep}\|_{L^1(\R^3)}^2\Big)   \\
&\le \f12 \|\na^3c^{l,\ep}\|_{L^2(\R^3)}^2  +  C\Big( \|\na \bfu^{l,\ep}\|_{L^2(\R^3)}^4  \|\na^2 c^{l,\ep}\|_{L^2(\R^3)}^2
   +  \|n_0^{l,\ep}\|_{L^1(\R^3)}^2\Big).
\end{align*}
In particular, a direct application of Gronwall's inequality can yield the uniform boundedness of  $\|\na c^{l,\ep}\|_{L^\infty\big(0,T; H^1(\R^3)\big)}$ in $T\in (0, T^*)$ if $T^*<+\infty$,   which contradicts to the Serrin condition in \eqref{blow-1} of $c^{l,\ep}$ provided that  we  take $r_1=6$ and $s_1=+\infty$.

Thus we have completed the proof of Proposition \ref{regular-global}.
\end{proof}

\subsection{Passing to the limit}\label{passlimit}

In this subsection, we shall extract a suitable subsequence from $(n^{l,\ep},c^{l,\ep},\bfu^{l,\ep})$ with  the help of \textit{a priori} energy estimates such that it is convergent and the corresponding limit triple $(n, c, \bfu)$ will be  a  global weak solution of  system \eqref{EQ}-\eqref{initial}.

\begin{proof}[Proof of Theorem \ref{global-weak}]
Firstly, for any $T\in (0, +\infty)$,  we can use the energy inequality \eqref{regular-entropy} to conclude the following uniform bounds:
\begin{enumerate}

  \item[($a$).] $ \|\sqrt{n^{l,\ep}}\|_{L^{\infty}\big(0,T; L^2(\R^3)\big)} + \big\|n^{l,\ep}\big|\ln n^{l,\ep}\big| \big\|_{L^{\infty}\big(0,T; L^1(\R^3)\big)}  + \|\na\sqrt{n^{l,\ep}}\|_{L^2\big(0,T; L^2(\R^3)\big)} \leq C$,  which together with the interpolation inequality
 \[
   \|n^{l,\ep}\|_{L^2(\R^3)}
  \le C \|\sqrt{n^{l,\ep}}\|_{L^2(\R^3)}^{\f12}  \|\nabla \sqrt{n^{l,\ep}}\|_{L^2(\R^3)}^{\f32}
  \le C\|\nabla \sqrt{n^{l,\ep}}\|_{L^2(\R^3)}^{\f32}
 \]
 implies that
  \[
  \|n^{l,\ep}\|_{L^{\f43}\big(0,T; L^2(\R^3)\big)} \le C
  \]
   and thus that
   \[
  \|\na n^{l,\ep}\|_{L^{\f87}\big(0,T; L^{\f43}(\R^3)\big)} \le C
  \]
 by the interpolation inequality
 \begin{align*}
   \|\na n^{l,\ep}\|_{L^{\f43}(\R^3)}
  & =  \|2\sqrt{n^{l,\ep}} \na \sqrt{n^{l,\ep}} \|_{L^{\f43}(\R^3)}   \\
 &  \le  2\|\sqrt{n^{l,\ep}}\|_{L^4(\R^3)}  \|\na \sqrt{n^{l,\ep}} \|_{L^2(\R^3)}   \\
 & \le   2\|n^{l,\ep}\|_{L^2(\R^3)}^{\f12} \|\na \sqrt{n^{l,\ep}} \|_{L^2(\R^3)};
\end{align*}

\item[($b$).] $\|c^{l,\ep}\|_{L^{\infty}\big(0,T; H^1(\R^3)\big)}  +  \|c^{l,\ep}\|_{L^2\big(0,T; H^2(\R^3)  \big)}\le C$;

\item[($c$).] $\|\bfu^{l,\ep}\|_{L^{\infty}\big(0,T; L^2(\R^3)\big)} + \|\na\bfu^{l,\ep}\|_{L^2\big(0,T; L^2(\R^3)\big)} \le C$;

\item[($d$).] $\|\pa_t n^{l,\ep}\|_{L^{\f43}\big(0,T; H^{-3}(\R^3)\big)}\le C$, which follows from property ($a$) and the estimate
    \begin{align*}
      \Big|\lt\langle\pa_tn^{l,\ep}, \varphi\rt\rangle\Big|
      & = \Big|\lt\langle n^{l,\ep}, \Del\varphi\rt\rangle + \lt\langle\bfu^{l,\ep}n^{l,\ep} + n^{l,\ep}\big(\chi(c^{\l,\ep})\na c^{l,\ep}\big) \ast\eta^{\ep}
        -  \big(n^{l,\ep} \na\phi\big) \ast\eta^{\ep}, \na\varphi\rt\rangle \Big|    \\
      &\le C\|n^{l,\ep}\|_{L^2(\R^3)} \Big(1 + \|\bfu^{l,\ep}\|_{L^2(\R^3)} + \|\na c^{l,\ep}\|_{L^2(\R^3)}\Big) \|\varphi\|_{H^3(\R^3)} \\
       & \le C\|n^{l,\ep}\|_{L^2(\R^3)} \|\varphi\|_{H^3(\R^3)}
  \end{align*}
      for any $\varphi\in H^3(\R^3)$;

\item[($e$).]  $\|\pa_t c^{l,\ep}\|_{L^{\f43}\big(0,T; L^2(\R^3)\big)} \le C$, which can be deduced from properties ($a$), ($b$) and  ($c$) together with  the estimate
\begin{align*}
  \|\pa_t c^{l,\ep}\|_{L^2(\R^3)}
  &\le  \|\bfu^{l,\ep}\|_{L^6(\R^3)}\|\na c^{l,\ep}\|_{L^3(\R^3)}  +  \|\Del c^{\l,\ep}\|_{L^2(\R^3)}  +  C\|n^{l,\ep}\|_{L^2(\R^3)}  \nn \\
  & \le C\|\na\bfu^{l,\ep}\|_{L^2(\R^3)} \|\na^2 c^{l,\ep}\|_{L^2(\R^3)}^{\f12}\|\na c^{l,\ep}\|_{L^2(\R^3)}^{\f12}
     +  \|\Delta c^{\l,\ep}\|_{L^2(\R^3)}  +  C\|n^{l,\ep}\|_{L^2(\R^3)} \\
 & \le C\|\na\bfu^{l,\ep}\|_{L^2(\R^3)}\|\na^2 c^{l,\ep}\|_{L^2(\R^3)}^{\f12}
     +  \|\na^2 c^{\l,\ep}\|_{L^2(\R^3)}  +  C\|n^{l,\ep}\|_{L^2(\R^3)};
 \end{align*}

\item[($f$).] $\|\pa_t\bfu^{l,\ep}\|_{L^{\f43}\big(0,T;  H^{-2}(\R^3)\big)} \leq C$, which can be seen from  the fact that for any $\bfw \in H^2(\R^3)$, it holds that
  \[
      \Big|\int_{\R^3}P_l(n^{l,\ep}\na\phi)\cd\bfw dx\Big|
      \le C\|n^{l,\ep}\|_{L^2(\R^3)}\|\bfw\|_{L^2(\R^3)}
  \]
      and
  \begin{align*}
       \Big| \int_{\R^3}P_l(n^{l,\ep}\chi(c^{l,\ep})\na c^{l,\ep})\cd\bfw  dx\Big|
      &\le C \|n^{l,\ep}\|_{L^2(\R^3)} \|\na c^{l,\ep}\|_{L^2(\R^3)}\|\bfw\|_{L^{\infty}(\R^3)}   \nn \\
      &\le C \|n^{l,\ep}\|_{L^2(\R^3)}  \|\bfw\|_{H^2(\R^3)}.
  \end{align*}
\end{enumerate}

Thus   for any $T\in (0, +\infty)$,  the above uniform estimates entail us to find a subsequence still denoted by $(n^{l,\ep}, c^{l,\ep}, \bfu^{l,\ep})$ for simplicity and some triple $(n, c, \bfu)$ with the regularities enquired in Definition \ref{def-weak} (i) such that when $l \to +\infty$ and $\ep \to 0$, it holds that

\begin{itemize}

 \item[($\widetilde{a}$).]  $\sqrt{n^{l,\ep}} \stackrel{\star}{\rightharpoonup} \sqrt{n}$ \, in $L^{\infty}\big(0, T; L^2(\R^3)\big)$, \,\,   $n^{l,\ep}\big|\ln n^{l,\ep}\big|  \stackrel{\star}{\rightharpoonup} n|\ln n|$ \, in $L^{\infty}\big(0, T; L^1(\R^3)\big)$,  \,\,  and    $\sqrt{n^{l,\ep}} \rightharpoonup \sqrt{n}$ \,  in $L^2\big(0, T; H^1(\R^3)\big)$ \, due to the property ($a$);

   \item[($\widetilde{b}$).]  $n^{l,\ep}\rightharpoonup n $ \, in $L^{\f43}\big(0, T; L^2(\R^3)\big)$ \, and \,  $\na n^{l,\ep}\rightharpoonup \na n $ \,  in $L^{\f87}\big(0, T; L^{\f43}(\R^3)\big)$ \, due to the property ($a$);

  \item[($\widetilde{c}$).]  $n^{l,\ep} \to n$ \, in $L^{\f87}\big(0, T; L^{\f{12}{5}}_{\mathrm{loc}}(\R^3)\big)$ \, by properties  ($a$) and  ($d$), and the Aubin–Lions Lemma with $W^{1,\f43}_{\mathrm{loc}}(\R^3) \subset\subset L^{\f{12}{5}}(\R^3) \subset H^{-3}(\R^3)$;

 \item[($\widetilde{d}$).]  $c^{l,\ep}  \stackrel{\star}{\rightharpoonup} c$  \,  in $L^{\infty}\big(0, T; H^1(\R^3)\big)$ \,\,  and \, $c^{l,\ep}  \rightharpoonup c$  \,  in $L^2\big(0, T; H^2(\R^3)\big)$ \, due to the property ($b$);

  \item[($\widetilde{e}$).]  $c^{l,\ep} \to c$ \, in $L^2\big(0, T; W^{1,6}_{\mathrm{loc}}(\R^3)\big)$ \, by using   properties  ($b$) and  ($e$), and  the Aubin–Lions Lemma;

 \item[($\widetilde{f}$).]  $\bfu^{l,\ep}  \stackrel{\star}{\rightharpoonup} \bfu$  \,   in $L^{\infty}\big(0, T; L^2(\R^3)\big)$ \,  and  \, $\bfu^{l,\ep} \rightharpoonup \bfu$  \,  in $L^2\big(0, T; H^1(\R^3)\big)$ \, due to the property ($c$);

\item[($\widetilde{g}$).]  $\bfu^{l,\ep} \to \bfu$ \, in $L^2\big(0, T; L^6_{\mathrm{loc}}(\R^3)\big)$ \, by  properties  ($c$) and  ($f$), and the Aubin–Lions Lemma.

\end{itemize}

\smallskip

Hence the above convergence properties together with Proposition \ref{operator-property}  entail us to  pass to the limit for all terms in the distributional form of system \eqref{regular-EQ} and thus  the triple $(n, c, \bfu)$ satisfies  Definition \ref{def-weak} (ii). We take  the convergence of the term
\[
\int_0^T\int_{\R^3}n^{l,\ep} \bfu^{l,\ep}\cd  \na\varphi dxdt
\]
for any given $\varphi \in C_0^{\infty}\big(\R^3\times[0, T]\big)$ satisfying $\varphi(\cd,T)=0$ as an example: supposing that $\mathrm{supp}\, \varphi (\cdot, t) \subset K$ for each $t\in[0, T]$, we can deduce that
\begin{align*}
& \int_0^T\int_{\R^3}n^{l,\ep} \bfu^{l,\ep}\cd  \na\varphi dxdt - \int_0^T\int_{\R^3}n \bfu \cd  \na\varphi dxdt     \\
& =  \int_0^T\int_{\R^3}n^{l,\ep} \big(\bfu^{l,\ep} - \bfu\big) \cd  \na\varphi dxdt  +  \int_0^T\int_{\R^3} \big(n^{l,\ep} - n\big) \bfu  \cd  \na\varphi dxdt \\
& \le C \int_0^T \big\|n^{l,\ep}\big\|_{L^{\f65}(K)} \big\| \bfu^{l,\ep} - \bfu\big\|_{L^6(K)}  + C \int_0^T \big\|n^{l,\ep} - n\big\|_{L^{\f{12}{5}}(K)}\|\bfu\|_{L^{\f{12}{7}}(K)} \\
& \le C \int_0^T \big\| \sqrt{n^{l,\ep}}\big\|_{L^{\f{12}5}(K)}^2 \big\| \bfu^{l,\ep} - \bfu\big\|_{H^1(K)}  + C \int_0^T \big\|n^{l,\ep} - n\big\|_{L^{\f{12}{5}}(K)} \|\bfu\|_{L^2(K)} \\
& \le C \int_0^T \big\| \sqrt{n^{l,\ep}}\big\|_{L^{2}(K)}^{\f32}  \big\| \na \sqrt{n^{l,\ep}}\big\|_{L^{2}(K)}^{\f12}  \big\| \bfu^{l,\ep} - \bfu\big\|_{H^1(K)}  + C \int_0^T \big\|n^{l,\ep} - n\big\|_{L^{\f{12}{5}}(K)} \|\bfu\|_{L^2(K)} \\
& \le C \int_0^T  \big\| \na \sqrt{n^{l,\ep}}\big\|_{L^{2}(K)}^{\f12}  \big\| \bfu^{l,\ep} - \bfu\big\|_{H^1(K)}  + C \int_0^T \big\|n^{l,\ep} - n\big\|_{L^{\f{12}{5}}(K)}   \\
& \le C T^{\f14}\Big(\int_0^T  \big\| \na \sqrt{n^{l,\ep}}\big\|_{L^{2}(K)}^2\Big)^{\f14} \Big(\int_0^T \big\| \bfu^{l,\ep} - \bfu\big\|_{H^1(K)}^2\Big)^{\f12}  + C T^{\f18}\Big(\int_0^T \big\|n^{l,\ep} - n\big\|_{L^{\f{12}{5}}(K)}^{\f87}\Big)^{\f78}    \\
& \to 0  \qquad \mathrm{as} \quad  l \to +\infty \,\,\,  \mathrm{and} \,\,\, \ep \to 0
\end{align*}
due to the uniform estimates  ($a$) and  ($c$), and the convergence properties  ($\widetilde{c}$) and  ($\widetilde{g}$), which implies that
\[
\int_0^T\int_{\R^3}n^{l,\ep} \bfu^{l,\ep}\cd  \na\varphi dxdt
\to \int_0^T\int_{\R^3}n \bfu \cd  \na\varphi dxdt   \qquad \mathrm{as} \quad  l \to +\infty \,\,\,  \mathrm{and} \,\,\, \ep \to 0.
\]
The other terms can be similarly dealt with.  This completes the proof of Theorem \ref{global-weak}.
\end{proof}

\vspace{2mm}

\section*{Acknowledgments}
JAC was supported by the Advanced Grant Nonlocal-CPD (Nonlocal PDEs for Complex Particle Dynamics: Phase Transitions, Patterns and Synchronization) of the European Research Council Executive Agency (ERC) under the European Union’s Horizon 2020 research and innovation programme (grant agreement No. 883363). JAC was also partially supported by EPSRC grants EP/T022132/1 and EP/V051121/1.
YP is partially supported by the Applied Fundamental Research Program of Sichuan Province (No. 2020YJ0264) and the Fundamental Research Funds for the Central Universities. ZX was supported by the NNSF of China (no. 11971093) and the Applied Fundamental Research Program of Sichuan Province (no. 2020YJ0264).

\end{document}